\documentclass[10pt,final]{siamltex}
\usepackage{amsmath}
\usepackage{bm}
\usepackage{amssymb,version}
\usepackage{cases}
\usepackage{color}
\usepackage{verbatim}
\usepackage{multirow}
\usepackage{graphicx}
\usepackage{subfigure}
\usepackage{graphics}
\usepackage{epsfig}
\usepackage{enumitem}
\newtheorem{remark}{Remark}[section]

\usepackage{silence}
\usepackage{hyperref}
\allowdisplaybreaks
\begin{document}
\graphicspath{{figures/},}
    \title{A linear mass-lumped finite element method for the Landau-Lifshitz-Gilbert equation: unconditional energy dissipation and length preservation
\thanks{This work is supported by the National Natural Science Foundation of China (Grant Nos. 12271302, 12131014) and Shandong Provincial Natural Science Foundation for Outstanding Youth Scholar (Grant No. ZR2024JQ030)}, and NSF (Grant No. DMS-2309548)}
 \author{Qiumei Huang 
      \thanks{College of Mathematics, Faculty of Science, Beijing University of Technology, Beijing 100124, China. Email: qmhuang@bjut.edu.cn}. 
       \and Binghong Li
        \thanks{School of Mathematics and Mathematical Research Center, Shandong University, Jinan, Shandong, 250100, P.R. China. Email: binghongsdu@163.com}.
        \and Xiaoli Li
        \thanks{Corresponding Author. School of Mathematics and State Key Laboratory of Cryptography and Digital Economy Security, Shandong University, Jinan, Shandong, 250100, P.R. China. Email: xiaolimath@sdu.edu.cn}.
        \and Cheng Wang
        \thanks{Mathematics Department, University of Massachusetts, North Dartmouth, MA 02747 USA. Email: cwang1@umassd.edu}.
        \and Jiang Yang
        \thanks{Department of Mathematics, Southern University of Science and Technology, Shenzhen 518055, China. Email: yangj7@sustech.edu.cn}. 
}

\maketitle
\begin{abstract}
We develop a linear, unconditionally energy-dissipative, mass-lumped finite element method for the highly nonlinear Landau--Lifshitz--Gilbert (LLG) equation on quasi-uniform triangular meshes. The method is built on a projection strategy for enforcing the nonconvex pointwise constraint $|\mathbf{m}| = 1$, whose simultaneous preservation with unconditional energy stability remains challenging for standard finite element discretizations. The key innovation is a unified hybrid finite element-finite difference framework that underlies both the design and the analysis of the proposed method. In the scheme construction, we exploit the weak formulation and nodal structure of mass-lumped finite element method, while incorporating suitable interpolation operators and a node-wise length-preserving mechanism inspired by finite difference discretizations. This combination yields a linear scheme that preserves the node-wise unit-length constraint and satisfies a discrete energy dissipation law. The same hybrid framework also plays a central role in the error analysis, where the weak formulation and quasi-uniform mesh structure of finite elements are combined with interpolation-based and nodewise finite difference method to control the strongly nonlinear damping term and to establish an optimal-order error estimate. More importantly, the proposed method provides a unified framework that systematically integrates the geometric flexibility of finite element method with the constraint-preserving property of finite difference method, and thus offers a general strategy for designing and analyzing structure-preserving discretizations of constrained dissipative systems. Numerical experiments, including a classical blow-up simulation, confirm the predicted accuracy, energy dissipation, and robustness of the method.

 \end{abstract}

  
 \section{Introduction}
 The Landau–Lifshitz–Gilbert (LLG) equation is widely used to characterize the magnetization dynamics in ferromagnetic materials \cite{jia2025electrically, jiang2015blowing, romming2013writing, nagaosa2013topological}, described by the form of a time-dependent nonlinear equation \cite{guo1993landau, landau1992theory}
 \begin{align} 
   \setlength{\abovedisplayskip}{4pt}  
   & 
     \mathbf{m}_t = - \beta\mathbf{m} \times \Delta\mathbf{m} - \gamma \mathbf{m} \times (\mathbf{m}\times\Delta\mathbf{m}),\,\, \text{in}\,\,\Omega\times J,  \label{originalmodel} 
 \\
   & 
     \mathbf{m}(\textbf{x},0) = \mathbf{m}_0(\textbf{x}),\,\,\text{with}\,\,|\mathbf{m}_0(\textbf{x})| = 1,\,\,\forall\,\textbf{x}\in \Omega,  \label{constraincondition} 
     \setlength{\belowdisplayskip}{4pt}  
 \end{align}
 subject to either homogeneous Neumann or periodic boundary conditions. Here $\mathbf{m}(\textbf{x},t) \in \mathbb{R}^3$ represents the magnetization vector, $\Omega$ is a bounded and convex domain in $\mathbb{R}^d$ with $d\in\{1,2,3\}$, $J$ denotes $(0,T]$ for some $T>0$, $\gamma>0$ is the Gilbert damping parameter and $\beta \in \mathbb{R}$ is an exchange parameter. Under condition \eqref{constraincondition}, it is known that the solution of \eqref{originalmodel} preserves its magnitude at a point-wise level, i.e.,
 \begin{equation}
   \setlength{\abovedisplayskip}{4pt}   
     |\mathbf{m}| = 1, \quad \text{in}\,\,\Omega\times J. 
  \setlength{\belowdisplayskip}{4pt}  
 \end{equation}
 Meanwhile, the following energy dissipation law also becomes valid: 
 \begin{equation}
    \setlength{\abovedisplayskip}{4pt}  
     \frac{dE}{dt} = - \gamma\|\mathbf{m}\times \Delta \mathbf{m}\|^2_{L^2},\,\,\text{where}\,\,E(\mathbf{m}) =\frac{1}{2}\int_\Omega|\nabla \mathbf{m}|^2d\textbf{x}. 
     \setlength{\belowdisplayskip}{4pt} 
 \end{equation}

 For energy-dissipative systems, simultaneously preserving energy dissipation and other essential physical properties has long been a central research focus \cite{du2021maximum, hou2024energy, hou2025energy}. As a well-known fact, it is essential to preserve both the energy dissipation property and the non-convex unit-length constraint $|\mathbf{m}| =1$ in the numerical simulation of the LLG equation. However, standard discretization techniques, particularly classical finite element methods (FEMs), typically fail to maintain an exact unit-length preservation \cite{akrivis2021higher}.  A substantial body of literature has investigated the convergence of unconstrained FEMs for the LLG equation without imposing renormalization \cite{akrivis2021higher, an2016optimal, cimrak2005error}. Among the various strategies developed to address the unit-length constraint, the sphere-projection or renormalization scheme has been widely adopted and demonstrated effectiveness in numerical computations. The renormalization approach \cite{cohen1989relaxation, gui2022convergence, kim2017mimetic, li2005numerical, suess2002time} enforces the constraint by normalizing the intermediate magnetization $\mathbf{\widetilde{m}}^n$ at each time step, setting $\displaystyle \textbf{m}^n = \mathbf{\widetilde{m}}^n/\,|\mathbf{\widetilde{m}}^n|.$ Despite its practical success, a rigorous error analysis of such a straightforward renormalization scheme remains challenging, especially in a combination with widely used finite element discretization and time-stepping schemes. The primary difficulty lies in establishing the stability of the discrete te,poral derivatives under the step-by-step renormalization process. This challenge is not unique to the LLG equation but is common across a broad class of related PDEs, including the harmonic map heat flow, the Landau–Lifshitz equation, and the nematic liquid crystal equations. In fact, there have been a series of works related to renormalization method. E and Wang established a first-order error estimate for the projection scheme in \cite{weinan2001numerical}.  An et al. \cite{an2021optimal} presented an error estimate for first- and second-order semi-implicit projection finite difference schemes, and the convergence result relies on inverse inequalities and requires a time step constraint, namely $h^2\le \Delta t\le h^{1+\epsilon_0}$. To relax this space-time constraint, an optimal-order error estimate has been established for a linearly implicit mass-lumped finite element method applied to the heat flow of harmonic maps on rectangular meshes in \cite{gui2022convergence}. This result holds under a constraint $ \Delta t \geq  \kappa h^{r+1} $ with $ r >1$, where $\kappa$ denotes an arbitrary positive constant.

To the authors' knowledge, none of the aforementioned normalization methods are unconditionally energy stable. To deal with this issue, Li et al. \cite{li2026stability} proposed a class of high-order implicit–explicit schemes based on the generalized scalar auxiliary variable approach for the Landau–Lifshitz equation. While the constructed schemes satisfy a modified energy dissipation law, a rigorous justification of the original energy dissipation property is still unavailable. Meanwhile, the normalized tangent plane FEM, which could also satisfy energy stability, was proposed by Alouges and Jaisson \cite{alouges2006convergence}. At each time step, this approach requires constructing a new finite element space that is node-wise orthogonal to the numerical solution obtained at the previous time step. Additionally, Alouges devised a discrete non-orthogonal projection scheme for an equivalent reformulation of \eqref{originalmodel}, namely 
 \begin{equation}
      \setlength{\abovedisplayskip}{4pt} 
        \lambda \textbf{m}_t - \textbf{m} \times \textbf{m}_t 
        = (1 + \lambda^2)(\Delta\textbf{m} + |\nabla \textbf{m}|^2\textbf{m}),\label{eq1.8}
     \setlength{\belowdisplayskip}{4pt} 
 \end{equation}
 where $\lambda > 0$ is a dimensionless Gilbert damping constant \cite{akrivis2021higher}, and demonstrated that this scheme preserves the energy dissipation and stability properties inherent to the continuous system \cite{alouges2008new}. Subsequently, An et al. \cite{an2025optimal} conducted a theoretical analysis of the normalized tangent plane method. While the renormalization approach provides a theoretical derivation of the original energy dissipation, it necessitates constructing a new finite element space at each time step, which incurs significant computational costs \cite{alouges2008new, an2025optimal}.

To develop a numerical scheme that is simple to implement and simultaneously guarantees unconditional energy dissipation and length preservation, inspired by the semi-implicit projection (SIP) method in \cite{du2025semi}, Li et al. \cite{li2026stability} has established a rigorous error analysis for this equivalent reformulation in the finite difference method (FDM), which preserves both the manifold constraint and the original energy dissipation:
\begin{equation}\label{semidiscrete}
  \setlength{\abovedisplayskip}{4pt} 
\left\{
    \begin{array}{l}
    \displaystyle
        \frac{\widetilde{\mathbf{m}}^{n+1} - \mathbf{m}^{n}}{\Delta t} = P(\mathbf{m}^n) \,
        \Delta_h \widetilde{\mathbf{m}}^{n+1} ,  
        \quad \mbox{with} \, \, \, P(\mathbf{m}) = \gamma(I-\frac{\mathbf{m}\mathbf{m}^T}{|\mathbf{m}|^2}) - \beta\mathbf{m}\times\cdot ,  \\
     \displaystyle    
       \mathbf{m}^{n+1} = \frac{\widetilde{\mathbf{m}}^{n+1}}{|\widetilde{\mathbf{m}}^{n+1}|} . 
    \end{array} \right. 
    \setlength{\belowdisplayskip}{4pt} 
\end{equation}
On the other hand, the error analysis relies heavily on the weak formulation constructed from the node-wise properties inherent in the finite difference discretization, which poses an essential difficulty in the extension to the finite element spatial approximation.

The main purpose of this paper is to present and analyze a fully discrete finite element scheme based on the SIP approach \eqref{semidiscrete}. Our main contributions are outlined as follows.

\begin{itemize}
    \item By integrating interpolation operators with the mass-lumping technique and leveraging the renormalization approach, we propose a fully discrete, linear finite element scheme for the LLG equation. The proposed scheme is straightforward to implement and enforces the non-convex constraint $|\mathbf{m}| = 1$ exactly, ensuring unconditional original energy dissipation.
     \item The constructed scheme could be interpreted as a hybrid approach that integrates the advantages of both FEMs and FDMs. In comparison with traditional FEMs, such an approach more effectively preserves node-wise properties, thereby inherently enhancing stability. In contrast, the constructed scheme accommodates more general quasi-uniform triangular partitions and demands lower regularity of the exact solution, which is advantageous over FDMs. 
    \item Owing to the node-wise length-preserving property and the equivalence of weak formulations, both of which are crucial for estimating the highly nonlinear damping terms, we establish a hybrid analytical framework to obtain optimal error analysis that combines mass-lumped finite element with finite difference expansion at each cell. This approach is primarily realized through the introduction of interpolation operators and a weak Laplacian operator. To the best of our knowledge, this is the first work to develop such an innovative hybrid finite element and finite difference method aimed at achieving optimal error estimates {\color{black} on quasi‑uniform triangular meshes}.
    

\end{itemize}

{\color{black} More importantly, the proposed method provides a unified framework that systematically integrates the geometric flexibility of finite element method with the constraint-preserving mechanisms of finite difference method, and thus offers a general strategy for designing and analyzing structure-preserving discretizations of constrained dissipative systems. 
We establish an optimal rate error estimate for the numerical solution, in the \(L^\infty(0,\,T;\,L^2(\Omega))\) and \(L^2(0,\,T;H^1(\Omega))\) norms. These results hold under mild grid ratio conditions: in two-dimensional (2D) cases with \(h^2 \lesssim \Delta t \lesssim h^{\epsilon_0}\) and in three-dimensional (3D) cases with \(h^2 \lesssim \Delta t \lesssim h^{1+\epsilon_0}\) and \(h \leq h_0\), where $0<\epsilon_0<1$ and $h_0$ is a positive constant.  

This paper is organized as follows. In Section 2, some preliminary notations and interpolation estimates are reviewed. The fully discrete finite element scheme is presented, and an unconditional energy dissipation is proved in Section 3. In Section 4, we provide a rigorous and optimal rate error analysis for the proposed finite element scheme. In Section 5, some numerical experiments are carried out using the constructed scheme. Finally, some concluding remarks are made in Section 6.
  \section{Some preliminary notations and results}
In this section, we review some notations and preliminary results, which will be frequently used in the later analysis. Throughout the paper, we use $C$, with or without subscript, to denote a positive constant, which could have different values at different appearances.

The standard notations, $L^2(\Omega)$, $H^k(\Omega)$ and $W^{k,p}(\Omega)$, are used to denote the usual Sobolev spaces over $\Omega$. In particular, $\|\cdot\|_{L^2}$ and $(\cdot,\cdot)$ are used to denote the norm and the inner product in $L^2(\Omega)$, respectively. The vectors and vector spaces will be indicated by boldface type. Set $\Delta t = T/N,\,t^n = n\Delta t$. In this paper, denote $\mathcal{K}_h$ as a quasi-uniform partition of $\Omega$ into triangles $K_j,\,\,j=1,\cdots,|\mathcal{K}_h|$ with mesh size $h = \max_{1\leq j\leq |\mathcal{K}_h|}\{diam K_j\}$, {\color{black} where, except for a finite number of triangles, the triangulation of $\Omega$ can be partitioned into pairs of adjacent triangles, each pair forming an $O(h^{2})$-approximate parallelogram, see \cite{bank2003asymptotically} for details}. We primarily consider \textcolor{black}{homogeneous Neumann} and periodic boundary conditions. \textcolor{black}{There is no essential difference in the treatment of these two types of boundary conditions in this paper. Therefore, this work mainly takes the periodic boundary condition as an example to present the numerical scheme and theoretical analysis. Unless otherwise specified, all function spaces in this paper correspond to the periodic boundary condition.} Subsequently, the tensor-product finite element space of degree $r\geq 1$ is defined as 
\begin{equation*}
    \mathbf{V}_h := \{\mathbf{v}_h \in \mathbf{H}^1:\mathbf{v}_h|_{K_j}\in {\color{black} \mathbf{P_1}}(K_j),\,\,\forall\,K_j\in \mathcal{K}_h\},
\end{equation*}
where $P_1(K_j)$ is the space of polynomials of degree $r=1$ on $K_j$. Let $\mathcal{I}_h$ be the piecewise Lagrange interpolation operator with all vertices $\mathcal{N}_h$ of the triangulation $\mathcal{K}_h$. In a lumped mass FEM, the discrete inner product $(\cdot,\cdot)_h$ is defined as follows, for any two continuous functions: 
\begin{equation*} 
  \setlength{\abovedisplayskip}{4pt}  
    (f,g)_h = \sum_{j=1}^{|\mathcal{K}_h|}\frac{|K_j|}{3}\sum^3_{i=1}f(x^{K_j}_i)g(x^{K_j}_i),\,\,\|f\|_{L^2_h}^2 = (f,f)_h, 
   \setlength{\belowdisplayskip}{4pt} 
\end{equation*} 
where $x^{K_j}_i$ stand for the three vertices of the small triangle $K_j$ respectively. More details of the mass-lumped finite element method could be found in \cite{chen1985lumped}. Moreover, for the finite element function space, the following estimate becomes available. 
\medskip
\begin{lemma}\label{lemlumped}\cite[Lemma 1]{chen1985lumped}
    Let $\varepsilon_h(v_h,w_h) = (v_h,w_h)_h-(v_h,w_h)$. We have for $v_h,w_h \in V_h$,
    \begin{equation}
     \setlength{\abovedisplayskip}{4pt}  
        |\varepsilon_h(v_h,w_h)| \leq Ch^2\|\nabla v_h\|_{L^2}\|\nabla w_h\|_{L^2}. 
     \setlength{\belowdisplayskip}{4pt}  
    \end{equation}
\end{lemma}
\medskip
\begin{lemma}\cite[Theorem 4.4.20]{brenner2008mathematical}
    The following error estimates are valid for the Lagrange interpolation operator $\mathcal{I}_h:C(\bar{\Omega}) \to \mathbf{V}_h$:
    \begin{equation}
    \setlength{\abovedisplayskip}{4pt}  
        \|\mathbf{v}-\mathcal{I}_h\mathbf{v}\|_{L^p}+h\|\mathbf{v}-\mathcal{I}_h\mathbf{v}\|_{W^{1,p}} \leq Ch^{s+1}\|\mathbf{v}\|_{W^{s+1,p}},
     \setlength{\belowdisplayskip}{4pt}  
    \end{equation}
\end{lemma}
for $0 \leq s\leq r$ and $p >d/(s+1)$ (in this case $W^{s+1,p}\hookrightarrow C(\bar{\Omega})$).

{\color{black} 
For ease of presentation, we describe the difference operators employed in our finite element method on a rectangular domain $\Omega = (0,L_x)\times(0,L_y)$, discretized by right-triangular meshes generated from quasi-uniform grids. For general domains, it suffices to impose the mesh conditions given in \cite{bank2003asymptotically}. Under these conditions, the extension to general triangulations follows in a natural manner.}

The two-dimensional domain $\Omega$ is partitioned by $\delta_x \times \delta_y$, where
\begin{equation*}
   \setlength{\abovedisplayskip}{4pt}   
    \delta_x:0=x_0<x_1<\cdots<x_{N_x}=L_x,\,\,\delta_y:0=y_0<y_1<\cdots<y_{N_y}=L_y . 
   \setlength{\belowdisplayskip}{4pt}  
\end{equation*}
Furthermore, the computational grid and the associated grid function space are given by 
\begin{equation*}
\aligned
\setlength{\abovedisplayskip}{4pt}   
    \mathcal{T}_h &= \{(i,j)|\,i=0,\cdots,N_x,\,j=0,\cdots,N_y\},\\
    \mathcal{M}_h &= \{m_{i,j}|\,(i,j) \in \mathcal{T}_h\},~~~~ \mathcal{M}^3_h = \mathcal{M}_h\times \mathcal{M}_h\times \mathcal{M}_h.
\endaligned
  \setlength{\belowdisplayskip}{4pt}  
\end{equation*}
Based on the above notations, for $f \in \mathcal{M}_h$, we denote
\begin{equation*}
  \setlength{\abovedisplayskip}{4pt}   
    \nabla^x_h f_{i+1/2,j} = \frac{f_{i+1,j} - f_{i,j}}{h_{i+1/2}},\,\nabla^y_h f_{i,j+1/2,k} = \frac{f_{i,j+1} - f_{i,j}}{\tau_{j+1/2}}, \, \,  h_{i+1/2} = x_{i+1}-x_i,\, \, \tau_{j+1/2} =y_{j+1}-y_j , 
    \setlength{\belowdisplayskip}{4pt}  
\end{equation*}
and $\nabla_hf = (\nabla_h^x f,\,\nabla^y_hf)^T$. In turn, for $\boldsymbol{f} = (f^x,f^y,f^z) \in \mathcal{M}_h^3$, the discrete gradient operator becomes 
\begin{equation*}
  \setlength{\abovedisplayskip}{4pt}  
    \nabla_h\boldsymbol{f} = 
    \left[ \begin{array}{cc}
         \nabla_h^x f^x& \nabla_h^y f^x \\
         \nabla_h^x f^y& \nabla_h^y f^y\\
         \nabla_h^x f^z& \nabla_h^y f^z
    \end{array}\right],
    \setlength{\belowdisplayskip}{4pt}  
\end{equation*}
and $\nabla_h\boldsymbol{f}:\nabla_h \boldsymbol{g}$ denotes the operation of multiplying the elements at the corresponding positions of the matrices, combined with a summation.

 The following results will be frequently employed in the subsequent analysis. Their proofs and more detailed discussions could be found in \cite{brenner2008mathematical}.
 \medskip
 \begin{lemma} (Inverse and interpolation inequalities)\label{leminverse}
     For any function $\boldsymbol{f}_h \in \mathbf{V}_h$, we have 
     \begin{align}
     \setlength{\abovedisplayskip}{4pt} 
     \displaystyle
         \|\boldsymbol{f}_h\|_{L^q} \leq Ch^{-(\frac{d}{p} - \frac{d}{q})}\|\boldsymbol{f}_h\|_{L^p}, \, \, \, 
         1\leq p < q  \leq \infty, \quad \text{$d$ is the dimension} ,  \label{inverse ineq-1} 
 \\
     \displaystyle
         \|\boldsymbol{f}_h \|_{L^q} \leq C  \|\boldsymbol{f}_h \|^{\frac{2}{q}} \cdot 
          \|\boldsymbol{f}_h \|_{H^1}^{1 - \frac{2}{q}} 
          \le \breve{C}_0  \|\boldsymbol{f}_h \|^{\frac{2}{q}} 
          ( \|\boldsymbol{f}_h \| + \| \nabla \boldsymbol{f}_h \| )^{1 - \frac{2}{q}} , 
          \, \, \, \forall\,2 < q < + \infty ,  \label{interpolation ineq-1} 
         \setlength{\belowdisplayskip}{4pt}
     \end{align} 
in which $\breve{C}_0$ is a constant only dependent on $\Omega$.
 \end{lemma}
 
Given any grid functions $f,\,g$ and vector grid functions $\boldsymbol{f},\,\boldsymbol{g}$, by the definition of $\nabla_h$, it is easy to verify that the following equalities hold:
 \begin{align*}
     \nabla_h (fg) &= \nabla_hf\,\mathcal{A}_hg + \nabla_hg\,\mathcal{A}_hf,\\
     \nabla_h(\boldsymbol{f} \cdot \boldsymbol{g}) &= \nabla_h\boldsymbol{f} \cdot \mathcal{A}_h\boldsymbol{g} + \nabla_h \boldsymbol{g} \cdot \mathcal{A}_h\boldsymbol{f},\\
     \nabla_h(f\boldsymbol{g}) &= \nabla_hf\,\mathcal{A}_h\boldsymbol{g} + \nabla_h\boldsymbol{g}\mathcal{A}_hf,\\
     \nabla_h(\boldsymbol{f}\times\boldsymbol{g}) &= \nabla_h\boldsymbol{f}\times\mathcal{A}_h\boldsymbol{g} + \mathcal{A}_h\boldsymbol{f}\times \nabla_h\boldsymbol{g},
 \end{align*}
 where \(\mathcal{A}_h\) is the matched central interpolation operator and the operations follow the tensor form.
 
The following discrete version of the Gronwall lemma (see, for instance, \cite{HeSu07,shen1990long}) will be frequently used. 
\medskip
\begin{lemma} \label{lem: gronwall2}
Let $a_k$, $b_k$, $c_k$, $d_k$, $\gamma_k$, $\Delta t_k$ be non-negative real numbers such that
\begin{equation*}\label{e_Gronwall3}
\aligned
a_{k+1}-a_k+b_{k+1}\Delta t_{k+1}+c_{k+1}\Delta t_{k+1}-c_k\Delta t_k\leq a_kd_k\Delta t_k+\gamma_{k+1}\Delta t_{k+1}
\endaligned
\end{equation*}
for all $0\leq k\leq m$. Then
 \begin{equation*}\label{e_Gronwall4}
 \setlength{\abovedisplayskip}{4pt}  
\aligned
a_{m+1}+\sum_{k=0}^{m+1}b_k\Delta t_k \leq \exp \Big(\sum_{k=0}^md_k\Delta t_k \Big)\{a_0+(b_0+c_0)\Delta t_0+\sum_{k=1}^{m+1}\gamma_k\Delta t_k \}.
\endaligned
 \setlength{\belowdisplayskip}{4pt} 
\end{equation*}
\end{lemma}
\section{An energy-dissipative and length preserving finite element scheme}
In this section, we construct a novel finite element numerical scheme that preserves both a node-wise unit length and the original energy dissipation. 

 First, for any $\mathbf{v}_h \in \mathbf{V}_h$, a weak Laplacian operator $\Delta_h:\,\mathbf{V}_h \to \mathbf{V}_h$ is introduced as  
\begin{equation} 
\setlength{\abovedisplayskip}{4pt}  
  ( \Delta_h \mathbf{v}_h, \mathbf{w}_ h )_h = -(\nabla \mathbf{v}_h, \nabla \mathbf{w}_h) , \quad 
  \forall \,\, \mathbf{v}_h , \, \mathbf{w}_h \in \mathbf{V}_h .  \label{A_h-def} 
\setlength{\belowdisplayskip}{4pt} 
\end{equation} 
Next, the fully discrete numerical scheme is proposed, based on a projection method in time and finite element spatial discretization.

\textbf{Scheme.} Given $\mathbf{m}_h^{n-1}$, find $\widetilde{\mathbf{m}}^{n}_h \in \mathbf{V}_h$ such that
\begin{equation}\label{discretescheme1}
  \setlength{\abovedisplayskip}{4pt}  
    \begin{aligned} 
      & 
         (\frac{\widetilde{\mathbf{m}}_h^n - \mathbf{m}_h^{n-1}}{\Delta t},\mathbf{v}_h)_h + \gamma(\nabla \widetilde{\mathbf{m}}_h^n,\nabla \mathbf{v}_h) \\
      & \qquad 
         = \beta (\nabla\widetilde{\mathbf{m}}^n_h,\nabla{\cal I}_h(\mathbf{v}_h \times \mathbf{m}^{n-1}_h)) + \gamma(\nabla
          \widetilde{\mathbf{m}}^n_h,\nabla{\cal I}_h((\mathbf{\mathbf{m}}^{n-1}_h \cdot \mathbf{v}_h)\mathbf{m}^{n-1}_h)),\,\,\forall\,\mathbf{v}_h \in \mathbf{V}_h , 
    \end{aligned}
   \setlength{\belowdisplayskip}{4pt} 
\end{equation}
where $\mathbf{m}^0_h = \mathcal{I}_h\mathbf{m}_e^0$. Afterward, the predicted numerical solution $\widetilde{\mathbf{m}}^n_h \in \mathbf{V}_h$ is renormalized as 
{\color{black}
\begin{equation}\label{discretescheme2}
   \setlength{\abovedisplayskip}{4pt}  
          \mathbf{m}^n_h(\textbf{x}_i) = \frac{\widetilde{\mathbf{m}}_h^{n}(\textbf{x}_i)}{|\widetilde{\mathbf{m}}_h^{n}(\textbf{x}_i)|}\Longleftrightarrow\mathbf{m}_h^n = {\cal I}_h \Big(\frac{\widetilde{\mathbf{m}}_h^{n}}{|\widetilde{\mathbf{m}}_h^{n}|} \Big),
   \setlength{\belowdisplayskip}{4pt}  
\end{equation}
which is simple to implement for the node‑wise normalization operation.
}

Subsequently, we derive the connection between the above scheme and the original LLG equation. By using the mass-lumped integration by parts formula \eqref{A_h-def}, it is observed that 
\begin{equation}
 \setlength{\abovedisplayskip}{4pt}  
\aligned 
         &(\frac{\widetilde{\mathbf{m}}_h^n - \mathbf{m}_h^{n-1}}{\Delta t},\mathbf{v}_h)_h - \gamma(\Delta_h \widetilde{\mathbf{m}}_h^n, \mathbf{v}_h)_h \\
         &~~~~= -\beta (\Delta_h\widetilde{\mathbf{m}}^n_h,{\cal I}_h(\mathbf{v}_h \times \mathbf{m}^{n-1}_h))_h - \gamma(\Delta_h
          \widetilde{\mathbf{m}}^n_h,{\cal I}_h((\mathbf{\mathbf{m}}^{n-1}_h \cdot \mathbf{v}_h)\mathbf{m}^{n-1}_h))_h \\ 
          &~~~~= -\beta ({\cal I}_h(\mathbf{m}^{n-1}_h\times \Delta_h\widetilde{\mathbf{m}}^n_h),\mathbf{v}_h)_h - \gamma ({\cal I}_h((\mathbf{\mathbf{m}}^{n-1}_h \cdot \Delta_h
          \widetilde{\mathbf{m}}^n_h)\mathbf{m}^{n-1}_h) , \mathbf{v}_h)_h . 
\endaligned          
  \setlength{\belowdisplayskip}{4pt} 
       \label{discrete1} 
\end{equation} 
Because of a vector identity $\boldsymbol{a}\times(\boldsymbol{b}\times \boldsymbol{c}) = (\boldsymbol{a}\cdot\boldsymbol{c})\boldsymbol{b}-(\boldsymbol{a}\cdot\boldsymbol{b})\boldsymbol{c}$, the following equation is obtained at the node points: 
\begin{equation*}
    \setlength{\abovedisplayskip}{4pt}   
    \widetilde{\mathbf{m}}^n_h({\bf x}_i) = \mathbf{m}^{n-1}_h({\bf x}_i)  - \Delta t \mathbf{m}^{n-1}_h \times (\beta \Delta_h\widetilde{\mathbf{m}}^{n}_h + \gamma \mathbf{m}_h^{n-1}\times \Delta_h\widetilde{\mathbf{m}}^{n}_h)({\bf x}_i),\,\, \, \forall\,1\leq i \leq |\mathcal{N}_h|,  
     \setlength{\belowdisplayskip}{4pt}  
\end{equation*}
which in turn implies that 
\begin{equation}
   \setlength{\abovedisplayskip}{4pt}   
    |\widetilde{\mathbf{m}}^{n}_{h}(\textbf{x}_i)|^2 \geq |\mathbf{m}^{n-1}_{h}(\textbf{x}_i)|^2=1, \quad 
    \widetilde{\mathbf{m}}^n_{h}(\textbf{x}_i)\cdot \mathbf{m}^{n-1}_{h}(\textbf{x}_i) = |\mathbf{m}^{n-1}_{h}(\textbf{x}_i)|^2 =1 .   
     \setlength{\belowdisplayskip}{4pt}  
\end{equation}
Therefore, the proposed numerical scheme could be rewritten as 
\begin{equation}\label{discretescheme}
   \setlength{\abovedisplayskip}{4pt}  
    \left\{
    \begin{array}{l}
    \displaystyle
         \frac{\widetilde{\mathbf{m}}_h^n - \mathbf{m}_h^{n-1}}{\Delta t} = -{\cal I}_h \Big( \beta \mathbf{m}_h^{n-1} \times \Delta_h\widetilde{\mathbf{m}}_h^{n} + \gamma\mathbf{m}_h^{n-1} \times (\mathbf{m}_h^{n-1} \times \Delta_h \widetilde{\mathbf{m}}_h^n) \Big) ,  \\
         \displaystyle
          \mathbf{m}_h^n = {\cal I}_h\bigg(\frac{\widetilde{\mathbf{m}}_h^{n}}{|\widetilde{\mathbf{m}}_h^{n}|}\bigg),
    \end{array} \right. 
    \setlength{\belowdisplayskip}{4pt} 
\end{equation}
which can be regarded to an approximation to the LLG equation, namely, $ \mathbf{m}_t = - \beta \mathbf{m}\times \Delta \mathbf{m}-\gamma \mathbf{m} \times (\mathbf{m} \times \Delta \mathbf{m})$. 

Next, an unconditional stability analysis is provided for the constructed fully discrete numerical scheme \eqref{discretescheme1}-\eqref{discretescheme2}. Before a formal proof, the following result turns out to be an important tool to derive the original energy dissipation law.
\medskip
\begin{lemma}\label{lem2}
    Let $\{\mathbf{m}^n_h,\widetilde{\mathbf{m}}_h^{n}\}$ be the solution of \eqref{discretescheme1}-\eqref{discretescheme2} with $ |\widetilde{\mathbf{m}}^{n}_{h\,i,j}|\geq1 $ and $ |\mathbf{m}^{n}_{h\,i,j}|= 1 $. For any $1\leq n \leq N$, the following node-wise estimate is valid: 
    \begin{equation}\label{lem2eq1}
        \setlength{\abovedisplayskip}{4pt}  
        | \nabla_h\mathbf{m}^n_{h} | \leq  |\nabla_h\widetilde{\mathbf{m}}^n_h | , \quad \forall \,\,(i, \,j) \in \mathcal{M}_h^3,
        \setlength{\belowdisplayskip}{4pt} 
    \end{equation}
    which in turn indicates that $\displaystyle \|\nabla \mathbf{m}^n_h\|_{L^2} \leq \|\nabla \widetilde{\mathbf{m}}^n_h\|_{L^2}$.
\end{lemma}
\begin{proof}
    For simplicity of presentation, we take \(|\nabla_h^x \mathbf{m}^n|^2\) as an example, and the same argument could be applied in the other directions. The following identity is observed: 
    \begin{equation*}
      \setlength{\abovedisplayskip}{4pt}  
        |\nabla^x_h\mathbf{m}^n_{h\,i+1/2,j}|^2 = |h_{i+1/2}^{-1}(\mathbf{m}^n_{h\,i+1,j} 
        - \mathbf{m}^n_{h\,i,j})|^2 . 
       \setlength{\belowdisplayskip}{4pt} 
    \end{equation*}
    Setting $\theta = \langle\mathbf{m}^n_{h\,i+1,j},~\mathbf{m}^n_{h\,i,j}\rangle,\,0 \leq \theta \leq \pi,~ \alpha_1 = |\widetilde{\mathbf{m}}^{n}_{h\,i+1,j}|\geq1,\,\alpha_2 = |\widetilde{\mathbf{m}}^{n}_{h\,i,j}| \geq 1$, and using the law of cosines, we see that 
    \begin{equation}\label{lem2eq2} 
      \setlength{\abovedisplayskip}{4pt}  
        h_{i+1/2}^2|\nabla^x_h\widetilde{\mathbf{m}}_{h\,i+1/2,j}^n|^2  = \alpha_1^2 + \alpha_2^2 - 2\alpha_1\alpha_2\cos\theta, \quad 
        h_{i+1/2}^2|\nabla_h^x\mathbf{m}_{h\,i+1/2,j}^n|^2 = 2 - 2\cos\theta.  
        \setlength{\belowdisplayskip}{4pt} 
    \end{equation}
    Taking the difference between the two quantities in \eqref{lem2eq2}, and using the fact that $\alpha_l \geq 1$ with $l=1,2$, it is clear that 
    \begin{equation*} 
    \setlength{\abovedisplayskip}{4pt}  
    \aligned
        &h_{i+1/2}^2(|\nabla_h^x\widetilde{\mathbf{m}}^n_{h\,i+1/2,j}|^2 - |\nabla_h^x\mathbf{m}^n_{h\,i+1/2,j}|^2)\\
        = &(\alpha_1^2 - 1) + (\alpha^2_2 - 1)-2\cos\theta(\alpha_1\alpha_2-1) 
        \ge  (\alpha_1 - \alpha_2)^2 \ge  0.
    \endaligned
    \setlength{\abovedisplayskip}{4pt}  
    \end{equation*}
    Subsequently, a summation over both directions leads to \eqref{lem2eq1}. Moreover, since $\mathbf{V}_h$ is a $P^1$ element space, we are able to get the desired inequality: 
    \begin{equation*}
    \setlength{\abovedisplayskip}{4pt}  
        \|\nabla \mathbf{m}^n_h\|_{L^2} \leq \|\nabla \widetilde{\mathbf{m}}^n_h\|_{L^2} . 
     \setlength{\belowdisplayskip}{4pt} 
    \end{equation*} 
    This finishes the proof of Lemma~\ref{lem2}. 
\end{proof}

With the help of the established bound in Lemma \ref{lem2}, we are able to derive an unconditional energy dissipation law and stability estimate.

\medskip
\begin{theorem}\label{stabilitythm}
    Let $\{\mathbf{m}^n,\widetilde{\mathbf{m}}^{n}\}$ be the solution of \eqref{discretescheme1}-\eqref{discretescheme2}. We have the following unconditional energy dissipation law, as well as a uniform $H^1$ bound for the numerical solution:  
    \begin{align} 
     \setlength{\abovedisplayskip}{4pt}   
      & 
         \frac{1}{2}\|\nabla \mathbf{m}^{n}_h\|^2_{L^2} -\frac{1}{2}\|\nabla \mathbf{m}^{n-1}_h\|^2_{L^2} \leq -\gamma \Delta t\|\mathbf{m}^{n-1}_h \times \Delta_h \widetilde{\mathbf{m}}^{n}_h\|^2_{L^2_h},\,\,\forall \,1\leq n\leq N ,  \\ 
      & \frac{1}{2}\|\nabla \mathbf{m}^l_h\|_{L^2}^2 +  \frac{1}{2}\|\nabla \mathbf{\widetilde{m}}_h^l\|_{L^2}^2 + \gamma\Delta t\sum^{l}_{n=1}\|\mathbf{m}_h^{n-1} \times \Delta_h\widetilde{\mathbf{m}}_h^n\|_{L^2_h}^2 \leq \|\nabla \mathbf{m}_h^0\|^2_{L^2}. 
      \setlength{\belowdisplayskip}{4pt}  
    \end{align}    
\end{theorem}
\begin{proof}
    Based on the identity $(\boldsymbol{a}\times \boldsymbol{b},\boldsymbol{b}) = 0$, it is clear that, taking $\mathbf{v}_h = -\Delta_h\widetilde{\mathbf{m}}^n_h$ in \eqref{discrete1} gives 
    \begin{equation*}
      \setlength{\abovedisplayskip}{4pt}   
        (\frac{\widetilde{\mathbf{m}}^n_h - \mathbf{m}^{n-1}_h}{\Delta t}, -\Delta_h\widetilde{\mathbf{m}}^{n}_h)_h = -\gamma (\mathcal{I}_h(\mathbf{m}^{n-1}_h \times (\mathbf{m}^{n-1}_h \times \Delta_h\widetilde{\mathbf{m}}^{n}_h)),-\Delta_h\widetilde{\mathbf{m}}^{n}_h)_h . 
      \setlength{\belowdisplayskip}{4pt}  
    \end{equation*}
    Moreover, by recalling \eqref{A_h-def} and Lemma \ref{lem2}, we get an original energy dissipation estimate: 
    \begin{equation*}
      \setlength{\abovedisplayskip}{4pt}  
        \frac{1}{2}\|\nabla \mathbf{m}^{n}_h\|^2_{L^2} \leq \frac{1}{2}\|\nabla \widetilde{\mathbf{m}}^n_h\|^2_{L^2} + \gamma \Delta t\|\mathbf{m}^{n-1}_h \times \Delta_h \widetilde{\mathbf{m}}^{n}_h\|^2_{L^2_h} \leq \frac{1}{2}\|\nabla \mathbf{m}^{n-1}_h\|^2_{L^2} . 
       \setlength{\belowdisplayskip}{4pt}  
    \end{equation*}
    A summation of this inequality from $n = 1$ to $n = l$ leads to the desired results.
\end{proof}

\textcolor{black}{The uniqueness of the solution to scheme \eqref{discretescheme1} can be naturally derived in the stability results of Theorem \ref{stabilitythm}, which implies the existence and uniqueness of the solution for the above scheme.
\medskip
\begin{remark}
It is observed that the stability results established in Theorem \ref{stabilitythm} could be extended to general triangular mesh partitions. As pointed out in \cite{alouges2008new}, the following inequality holds under the Delaunay-type condition in the two-dimensional case:
\begin{equation}
    \setlength{\abovedisplayskip}{4pt}  
    \int_\Omega \Big| \nabla \mathcal{I}_h \Big(\frac{\mathbf{v}_h}{|\mathbf{v}_h|} \Big) \Big|^2 \, d {\bf x} \leq \int_\Omega|\nabla\mathbf{v}_h|^2 \, d {\bf x} , \,\,\forall\,\mathbf{v}_h \in \mathbf{V}_h .  
    \setlength{\belowdisplayskip}{4pt}  
\end{equation}
This inequality ensures that the stability analysis remains valid on general unstructured triangular meshes.
\end{remark}
}

\section{Theoretical analysis}
In this section, we present an optimal rate convergence analysis and error estimate for the fully discrete scheme \eqref{discretescheme1}-\eqref{discretescheme2}. Such a theoretical analysis is very challenging, due to the highly complicated behavior inherent in the constructed scheme. In the theoretical derivation for  \eqref{discretescheme}, special care must be taken to handle the term $-\gamma\mathbf{m}^{n-1}\times(\mathbf{m}^{n-1}\times \Delta_h \widetilde{\mathbf{m}}^n)$. In this work, we first give an equivalent reformulation of the original numerical system, with the help of summation by parts, so that a highly complicated term is avoided. Afterward, an optimal rate convergence analysis is derived, based on the linearized stability estimates.

\subsection{Equivalent discrete scheme} For any vector function $\mathbf{v}_h \in \mathbf{V}_h$, \eqref{discretescheme1} gives 
\begin{equation}\label{scheme-alt-1} 
    \setlength{\abovedisplayskip}{4pt} 
    \aligned 
    &(\frac{\widetilde{\mathbf{m}}_h^{n} - \mathbf{m}_h^{n-1}}{\Delta t},\mathbf{v}_h )_h + \gamma(\nabla \widetilde{\mathbf{m}}_h^n, \nabla \mathbf{v}_h ) \\
    = & \beta ( \nabla \widetilde{\mathbf{m}}_h^n ,  \nabla ( {\cal I}_h (\mathbf{v}_h \times \mathbf{m}_h^{n-1}) ) )   + \gamma (\nabla \mathbf{m}_h^{n-1} + \nabla (\widetilde{\mathbf{m}}_h^n-\mathbf{m}_h^{n-1}) , \nabla ( {\cal I}_h ( (\mathbf{m}_h^{n-1}\cdot \mathbf{v}_h ) \mathbf{m}_h^{n-1} ) )  , 
    \endaligned 
    \setlength{\belowdisplayskip}{4pt}  
\end{equation} 
in which the definition~\eqref{A_h-def} for the weak Laplacian operator $\Delta_h$ has been repeatedly applied, and the interpolation operator has played an essential role in the derivation. 

Moreover, since a mass lumping element is used, for $f_h$, $g_h \in \mathbf{V}_h$, we would like to express their $H^1$ inner product in terms of the associated finite difference approximations at each element cell. Such an expression will play an important role in the later analysis. For simplicity of presentation, we focus on a right triangle element, with a right angle located at $(i, j)$, and two other vertices at $(i+1, j)$, $(i, j+1)$, respectively. The other triangular element could be similarly analyzed. If $f_h$ and $g_h$ have interpolation values of $f_{i,j}$, $f_{i+1,j}$, $f_{i,j+1}$, $g_{i,j}$, $g_{i+1,j}$, $g_{i,j+1}$ at the vertices, the $H^1$ inner product over the triangular cell $K_{i,j}$, as well as the one over the whole domain, turns out to be 
\begin{equation} 
\setlength{\abovedisplayskip}{4pt}  
\begin{aligned} 
  & 
  ( \nabla f_h, \nabla g_h )_{K_{i,j}} = | K_{i,j} | ( ( \nabla^x_h f)_{i+\frac12,j} (\nabla^x_h g)_{i+\frac12, j} 
  + ( \nabla^y_h f)_{i, j+\frac12} (\nabla^y_h g)_{i, j+\frac12}  )  ,    
\\ 
   & 
   ( \nabla f_h, \nabla g_h ) = \sum_{K \in \mathcal{K}_h} ( \nabla f_h, \nabla g_h )_{K}  . 
\end{aligned} 
\setlength{\belowdisplayskip}{4pt}  
  \label{H1 inner product-1} 
\end{equation} 
In turn, the following finite difference operators are introduced at each triangular mesh area, to simplify the notations in the later sections: 
\begin{equation} 
\setlength{\abovedisplayskip}{4pt}  
\begin{aligned} 
  & 
  ( \nabla_h f_h \cdot \nabla_h g_h )_{K_{i,j}}:= ( \nabla^x_h f_h  \nabla^x_h g_h )_{i+\frac12, j} 
  + ( \nabla^y_h f_h \nabla^y_h g_h )_{i, j+\frac12}  ,  
\\
  &  
  ( \nabla f_h, \nabla g_h )_{K_{i,j}}  = | K_{i,j} |  ( \nabla_h f_h \cdot \nabla_h g_h )_{K_{i,j}} . 
\end{aligned} 
\setlength{\belowdisplayskip}{4pt}  
  \label{H1 inner product-2} 
\end{equation} 

On the other hand, since the interpolation operator has been deeply involved in the evaluation of the nonlinear inner product, a detailed expansion of $\nabla ( {\cal I}_h ( f_h g_h))$ is needed, for $f_h$, $g_h \in \mathbf{V}_h$. It is clear that ${\cal I}_h ( f_h g_h)$ is the interpolated function in the finite element space, with three vertices values of $(f g)_{i,j}$, $(f g)_{i+1,j}$, $(f g)_{i,j+1}$. In turn, its gradient over the cell $K_{i,j}$ becomes a constant vector: 
\begin{equation} 
\begin{aligned} 
  & 
  \partial_x ( {\cal I}_h ( f_h g_h) ) \equiv \frac{(fg)_{i+1, j} - (fg)_{i,j}}{h_{i+1/2}} = \nabla^x_h (fg)_{i+\frac12, j} 
  = {\cal A}_x f_{i+\frac12, j} \nabla^x_h g_{i+\frac12, j} + {\cal A}_x g_{i+\frac12, j} \nabla^x_h f_{i+\frac12, j} , 
\\
  & 
  \partial_y ( {\cal I}_h ( f_h g_h) ) \equiv \frac{(fg)_{i, j+1} - (fg)_{i,j}}{\tau_{j+1/2}} = \nabla^y_h (fg)_{i, j+\frac12} 
  = {\cal A}_y f_{i, j+\frac12} \nabla^y_h g_{i, j+\frac12} + {\cal A}_y g_{i, j+\frac12} \nabla^y_h f_{i, j+\frac12} . 
\end{aligned} 
  \label{nonlinear expansion-1} 
\end{equation} 

Since both $\nabla \widetilde{\mathbf{m}}_h^n$ and $\nabla ( {\cal I}_h (\mathbf{v}_h \times \mathbf{m}_h^{n-1}) )$ are constant vectors over the triangular cell $K_{i,j}$, the nonlinear inner product for the first term on the right hand side of~\eqref{scheme-alt-1} could be expanded as follows, over the triangular cell $K_{i,j}$: 
\begin{equation} 
  \setlength{\abovedisplayskip}{4pt}  
\begin{aligned} 
  & 
   ( \nabla \widetilde{\mathbf{m}}_h^n ,  \nabla ( {\cal I}_h (\mathbf{v}_h \times \mathbf{m}_h^{n-1}) ) )_{K_{i,j}} 
   = | K_{i,j} | ( \nabla_h \widetilde{\mathbf{m}}^n \cdot \nabla_h ( \mathbf{v}_h \times \mathbf{m}_h^{n-1}) )_{K_{i,j}} 
\end{aligned} 
  \setlength{\belowdisplayskip}{4pt}  
   \label{nonlinear expansion-3-1} 
\end{equation} 
In more details, based on the finite difference expansion formula~\eqref{nonlinear expansion-1} for the nonlinear product over a triangular mesh, we observe that the discrete gradient operator could be expressed as follows: 
\begin{equation} 
  \setlength{\abovedisplayskip}{4pt}  
\begin{aligned} 
  & 
   \nabla_h ( \mathbf{v}_h \times \mathbf{m}^{n-1}_h)_{K_{i,j}}  
   = \left( \begin{array}{l} 
    \nabla^x_h ( \mathbf{v}_h \times \mathbf{m}^{n-1}_h)_{i+\frac12,j} \\ 
    \nabla^y_h ( \mathbf{v}_h \times \mathbf{m}^{n-1}_h)_{i,j+\frac12} 
    \end{array} \right)   
    = ( {\cal A}_h \mathbf{v}_h \times \nabla_h \mathbf{m}_h^{n-1} 
    + \nabla_h \mathbf{v}_h \times {\cal A}_h \mathbf{m}_h^{n-1} )_{K_{i,j}} 
\\
  = & 
    \left( \begin{array}{l} 
    ( {\cal A}_x \mathbf{v}_h \times \nabla^x_h \mathbf{m}_h^{n-1} 
    + \nabla^x_h \mathbf{v}_h \times {\cal A}_x \mathbf{m}^{n-1}_h )_{i+\frac12,j} \\ 
    ( {\cal A}_y \mathbf{v}_h \times \nabla^y_h \mathbf{m}_h^{n-1} 
    + \nabla^y_h \mathbf{v}_h \times {\cal A}_y \mathbf{m}^{n-1}_h)_{i,j+\frac12}  
    \end{array} \right)  ,
\end{aligned} 
  \setlength{\belowdisplayskip}{4pt}  
\end{equation}
so that
\begin{equation} 
  \setlength{\abovedisplayskip}{2pt}  
\begin{aligned}
  & 
  ( \nabla \widetilde{\mathbf{m}}_h^n ,  \nabla ( {\cal I}_h (\mathbf{v}_h \times \mathbf{m}_h^{n-1}) ) )  
   = \sum_{i,j} | K_{i,j} | ( \nabla_h \widetilde{\mathbf{m}}_h^n \cdot \nabla_h ( \mathbf{v}_h \times \mathbf{m}^{n-1}_h) )_{K_{i,j}} 
\\
  = & 
  \sum_{K \in \mathcal{K}_h} | K| ( \nabla_h \widetilde{\mathbf{m}}_h^n \cdot 
  ( {\cal A}_h \mathbf{v}_h \times \nabla_h \mathbf{m}_h^{n-1} 
    + \nabla_h \mathbf{v}_h \times {\cal A}_h \mathbf{m}^{n-1}_h ) )_{K} . 
\end{aligned} 
 \setlength{\belowdisplayskip}{2pt}  
   \label{nonlinear expansion-3-2} 
\end{equation} 

The other two nonlinear inner product terms could be similarly expanded in terms of the finite difference approximation over each finite element cell: 
\begin{equation} 
  \setlength{\abovedisplayskip}{4pt}  
\begin{aligned} 
  & 
  (\nabla \mathbf{m}_h^{n-1} , 
  \nabla ( {\cal I}_h ( (\mathbf{m}_h^{n-1}\cdot \mathbf{v}_h ) \mathbf{m}_h^{n-1} ) )  
   = \sum_{i,j} | K_{i,j} | (\nabla_h \mathbf{m}^{n-1}_h \cdot 
  \nabla_h ( (\mathbf{m}^{n-1}_h\cdot \mathbf{v}_h ) \mathbf{m}_h^{n-1} ) )_{K_{i,j}}  
\\
  = & 
     \sum_{K \in \mathcal{K}_h} |K| (\nabla_h \mathbf{m}^{n-1}_h \cdot  
   ( {\cal A}_h (\mathbf{m}_h^{n-1}\cdot \mathbf{v}_h ) \nabla_h \mathbf{m}^{n-1}_h 
    + \nabla_h (\mathbf{m}_h^{n-1}\cdot \mathbf{v}_h )  {\cal A}_h \mathbf{m}^{n-1}_h ) )_{K} , 
\end{aligned} 
  \setlength{\belowdisplayskip}{4pt}  
   \label{nonlinear expansion-4} 
\end{equation} 
\begin{equation}
  \setlength{\abovedisplayskip}{4pt}   
\begin{aligned} 
  & 
  ( \nabla (\widetilde{\mathbf{m}}_h^n-\mathbf{m}_h^{n-1}) , 
  \nabla ( {\cal I}_h ( (\mathbf{m}_h^{n-1}\cdot \mathbf{v}_h ) \mathbf{m}_h^{n-1} ) ) 
\\
  = & 
    \sum_{K \in \mathcal{K}_h} |K| (\nabla_h ( \widetilde{\mathbf{m}}^n_h - \mathbf{m}^{n-1}_h ) \cdot  
   ( {\cal A}_h (\mathbf{m}^{n-1}_h\cdot \mathbf{v}_h ) \nabla_h \mathbf{m}^{n-1}_h 
    + \nabla_h (\mathbf{m}_h^{n-1}\cdot \mathbf{v}_h )  {\cal A}_h \mathbf{m}^{n-1}_h ) )_{K} . 
\end{aligned} 
  \setlength{\belowdisplayskip}{4pt}  
   \label{nonlinear expansion-5} 
\end{equation} 

Meanwhile, two nonlinear inner product terms in the above expansions would either disappear or be transformed into a rewritten form, as stated in the following lemma. These two equalities will play a crucial role in the convergence analysis.

\medskip
\begin{lemma}\label{lem4}
    Assume $\{\mathbf{m}_h^n,\widetilde{\mathbf{m}}^{n}_h\}$ is the solution of \eqref{discretescheme1}-\eqref{discretescheme2}, satisfying  
    $$|\mathbf{m}_{h}^{n-1}|^2=1,\,\,\mathbf{m}^{n-1}_h \cdot (\widetilde{\mathbf{m}}^n_h -\mathbf{m}^{n-1}_h) =0,\,\,\text{at each node point.}$$ 
     Then the following equalities are valid, at each half-point $(i+\frac{1}{2},j)$ or $(i,j+\frac{1}{2})$:  
    \begin{align}
      \setlength{\abovedisplayskip}{4pt}  
        {\cal A}_h\mathbf{m}_h^n \cdot(\nabla_h \mathbf{m}_h^n) &=  0 , 
        \label{lem4eq1}\\
        {\cal A}_h\mathbf{m}_h^{n-1} \cdot \nabla_h(\widetilde{\mathbf{m}}_h^n -\mathbf{m}_h^{n-1}) 
       &=  - {\cal A}_h(\widetilde{\mathbf{m}}_h^n -\mathbf{m}_h^{n-1})\cdot\nabla_h\mathbf{m}^{n-1}_h .  
       \label{lem4eq2} 
       \setlength{\belowdisplayskip}{4pt}  
    \end{align}
\end{lemma}
\begin{proof}
    Since $|\mathbf{m}_h^n|=1$ at each node point, we see that 
    \begin{equation*}
      \setlength{\abovedisplayskip}{4pt}  
        {\cal A}_x \mathbf{m}^n_{h\,i+\frac12, j} \cdot \nabla^x_h \mathbf{m}^n_{h\,i+\frac12, j} = \frac{\mathbf{m}^n_{h\,i+1,j}+\mathbf{m}^n_{h\,i,j}}{2} \cdot\frac{\mathbf{m}^n_{h\,i+1,j}-\mathbf{m}^n_{h\,i,j}}{h_{i+1/2}} = \frac{|\mathbf{m}^n_{h\,i+1,j}|^2 - |\mathbf{m}^n_{h\,i,j} |^2}{2h_{i+1/2}}=0. 
        \setlength{\belowdisplayskip}{4pt}  
    \end{equation*}
A similar calculation also implies that ${\cal A}_y \mathbf{m}^n_{i, j+\frac12} \cdot \nabla^y_h \mathbf{m}^n_{i, j+\frac12}=0$, at each mesh cell. Therefore, the desired identity \eqref{lem4eq1} has been proved. Regarding the proof of \eqref{lem4eq2}, we see that the node-wise identity, $\mathbf{m}_h^{n-1} \cdot (\widetilde{\mathbf{m}}_h^n -\mathbf{m}_h^{n-1}) =0$ (at each node point) indicates the following expansion 
    \begin{equation*} 
    \setlength{\abovedisplayskip}{4pt}  
    \begin{aligned} 
     &\nabla^x_h (\mathbf{m}_h^{n-1}\cdot(\mathbf{\widetilde{\mathbf{m}}}_h^n - \mathbf{m}^{n-1}_h))_{i+\frac12, j}  
\\
  & =  ( {\cal A}_x \mathbf{m}_h^{n-1} \cdot \nabla^x_h (\widetilde{\mathbf{m}}_h^n -\mathbf{m}_h^{n-1}) )_{i+\frac12, j} 
  + ( {\cal A}_x (\widetilde{\mathbf{m}}_h^n -\mathbf{m}_h^{n-1}) \cdot \nabla^x_h \mathbf{m}_h^{n-1} )_{i+\frac12, j}= 0.
    \end{aligned} 
    \setlength{\belowdisplayskip}{4pt}  
    \end{equation*}
The derivative in the y-direction is treated analogously. In turn, a summation of these two equalities gives~\eqref{lem4eq2}, which finishes the proof of Lemma~\ref{lem4}.
\end{proof}

With the help of the two identities in Lemma~\ref{lem4}, a substitution of~\eqref{nonlinear expansion-3-2}-\eqref{nonlinear expansion-5} into \eqref{scheme-alt-1} results in 
\begin{equation}\label{m_hconsistency} 
   \setlength{\abovedisplayskip}{4pt}  
    \begin{aligned}
    \setlength{\abovedisplayskip}{4pt}  
    &(\frac{\widetilde{\mathbf{m}}_h^{n} - \mathbf{m}_h^{n-1}}{\Delta t},\mathbf{v}_h )_h + \gamma(\nabla \widetilde{\mathbf{m}}_h^n, \nabla \mathbf{v}_h ) \\ 
    = &  \sum_{K\in \mathcal{K}_h} | K | \Big( \beta ( \nabla_h \widetilde{\mathbf{m}}_h^n \cdot 
  ( {\cal A}_h \mathbf{v}_h \times \nabla_h \mathbf{m}_h^{n-1} 
    + \nabla_h \mathbf{v}_h \times {\cal A}_h \mathbf{m}_h^{n-1} ) )_{K}  \\ 
     & 
     + \gamma (\nabla_h \mathbf{m}^{n-1}_h \cdot  
   ( {\cal A}_h (\mathbf{m}^{n-1}_h\cdot \mathbf{v}_h ) \nabla_h \mathbf{m}_h^{n-1} ) 
   + \nabla_h ( \widetilde{\mathbf{m}}_h^n - \mathbf{m}_h^{n-1} ) \cdot  
   ( {\cal A}_h (\mathbf{m}^{n-1}_h\cdot \mathbf{v}_h ) \nabla_h \mathbf{m}^{n-1}_h )  )_{K}  \\
   & 
     - \gamma ( ( {\cal A}_h ( \widetilde{\mathbf{m}}_h^n - \mathbf{m}_h^{n-1} ) \cdot  
     \nabla_h \mathbf{m}_h^{n-1} )  \cdot  \nabla_h (\mathbf{m}_h^{n-1}\cdot \mathbf{v}_h ) ) _{K} \Big)  \\
     = &  \sum_{K\in\mathcal{K}_h} |K| \Big( \beta ( \nabla_h \widetilde{\mathbf{m}}_h^n \cdot 
  ( {\cal A}_h \mathbf{v}_h \times \nabla_h \mathbf{m}_h^{n-1} 
    + \nabla_h \mathbf{v}_h \times {\cal A}_h \mathbf{m}_h^{n-1} ) )_{K}  \\ 
     & 
     + \gamma \big( \nabla_h \widetilde{\mathbf{m}}_h^n \cdot  
   ( {\cal A}_h ( \mathbf{m}^{n-1}_h \cdot \mathbf{v}_h ) \nabla_h \mathbf{m}^{n-1}_h )   
    - ( {\cal A}_h ( \widetilde{\mathbf{m}}^n_h - \mathbf{m}^{n-1}_h ) \cdot  
     \nabla_h \mathbf{m}^{n-1}_h )  \cdot  \nabla_h (\mathbf{m}^{n-1}_h\cdot \mathbf{v}_h ) \big)_{K} \Big) .  
    \setlength{\belowdisplayskip}{4pt}  
    \end{aligned}
    \setlength{\belowdisplayskip}{4pt}  
\end{equation}
\subsection{Consistency analysis}
In this subsection, we derive the consistency of \eqref{m_hconsistency}, assuming a smooth exact solution. Set $\mathbf{m}_e$ as the exact solution for the PDE system, and $\mathbf{m}_{e, h} = {\cal I}_h \mathbf{m}_e$ as the interpolation of the exact solution in the finite element space. The original LLG model is equivalent to
\begin{equation} \label{originaleq} 
  \setlength{\abovedisplayskip}{4pt}   
    \begin{aligned}
        \partial_t\mathbf{m}_e - \gamma \Delta \mathbf{m}_e = -\beta\mathbf{m}_e\times \Delta \mathbf{m}_e -\gamma(\mathbf{m}_e\cdot\Delta\mathbf{m}_e)\mathbf{m}_e.
    \end{aligned}
   \setlength{\belowdisplayskip}{4pt}   
\end{equation}
Taking an inner product with \eqref{originaleq} by $\mathbf{v}_h$, and expanding it at the $n$-th layer, we get 
\begin{equation}
   \setlength{\abovedisplayskip}{4pt}    
    (\partial_t\mathbf{m}^n_e,\mathbf{v}_h) + \gamma (\nabla\mathbf{m}_e^n,\nabla\mathbf{v}_h) = -\beta(\Delta\mathbf{m}^n_e,\mathbf{v}_h\times \mathbf{m}^{n-1}_e)-\gamma(\Delta\mathbf{m}^{n}_e,(\mathbf{m}^{n-1}_e\cdot\mathbf{v}_h)\mathbf{m}_e^{n-1}) + \mathcal{E}^n_1(\mathbf{m}_e,\mathbf{v}_h), 
     \setlength{\belowdisplayskip}{4pt}   
\end{equation}
in which $\mathcal{E}_1(\mathbf{m}_e,\mathbf{v}_h)$ is bounded by
\begin{equation}\label{E1}
   \setlength{\abovedisplayskip}{4pt}   
    \Delta t \sum^{\ell}_{n=1}|\mathcal{E}^n_1(\mathbf{m}_e,\mathbf{v}_h)| \leq \epsilon\Delta t\sum^{\ell}_{n=1}\|\mathbf{v}^n_h\|^2_{H^1}+C(\| \mathbf{m}_e \|_{L^\infty(0,T;\boldsymbol{H}^2(\Omega))},\|\partial_t \mathbf{m}_e\|_{L^\infty(0,T;\boldsymbol{H}^2(\Omega))})\Delta t^2. 
     \setlength{\belowdisplayskip}{4pt}   
\end{equation}

Before proceeding with the proof of consistency and obtain an optimal rate convergence analysis, we need the following preliminary estimates.
\medskip
{\color{black}
\begin{lemma}\label{superconvergence}
     \cite[Lemma 2.5]{bank2003asymptotically}Let the triangulation $\mathcal{K}_h$ be $O(h^{2})$ irregular. For $\boldsymbol{u} \in \mathbf{W}^{3,\infty}(\Omega)$ and $\boldsymbol{v}_h \in \mathbf{V}_h$, the following super-convergence result holds for the Lagrange interpolation operator $\mathcal{I}_h$:
    \begin{equation}\label{lemsupeq}
      \setlength{\abovedisplayskip}{4pt}   
        |(\nabla(\boldsymbol{u}-\mathcal{I}_h\boldsymbol{u}),\nabla\boldsymbol{v}_h)| \lesssim h^{2}|\log h|^{1/2}\|\boldsymbol{u}\|_{W^{3,\infty}}\|\boldsymbol{v}_h\|_{H^1}. 
        \setlength{\belowdisplayskip}{4pt}   
    \end{equation}
\end{lemma}

Lemma \ref{superconvergence} stands for a super-convergence result for the Lagrange interpolation, which plays an important role in the truncation error estimate. Meanwhile, since our analysis relies on Lemma \ref{lem4} from the finite difference approach, it is insufficient to only use Lemma \ref{superconvergence} under a condition $\boldsymbol{v}_h \in \mathbf{V}_h$. Therefore, the following result is essential for a precise estimate for the introduced interpolation operator.
\medskip
\begin{lemma}\label{lem4.3}
    For $u \in W^{2,\infty}(\Omega)$ and $v_h \in V_h$ where \(V_h\) denotes the space of piecewise‑linear functions over \(\mathcal{K}_h\), the following interpolation result holds for the Lagrange interpolation operator $\mathcal{I}_h$:
    \begin{equation}
     \setlength{\abovedisplayskip}{4pt}   
        \|u v_h - \mathcal{I}_h(u v_h)\|_{L^1} \leq Ch^2\|u\|_{W^{2,\infty}}\|v_h\|_{H^1}.
       \setlength{\belowdisplayskip}{4pt}   
    \end{equation}
\end{lemma}
\begin{proof}
    Over each triangular mesh $K\in \mathcal{K}_h$, the Lagrange interpolation estimate gives 
    \begin{equation}\label{eq4.18} 
     \setlength{\abovedisplayskip}{4pt}   
        \|u v_h - \mathcal{I}_h(u v_h)\|_{L^1(K)} \leq Ch^2\|u v_h\|_{W^{2,1}(K)}, 
        \setlength{\belowdisplayskip}{4pt}   
    \end{equation}
    where $|\textbf{v}_h|_{H^2(K)} = 0$. In turn, the following bound becomes available: 
    \begin{equation}\label{eq4.19}
     \setlength{\abovedisplayskip}{4pt}   
    \begin{aligned}
         \|u v_h\|_{W^{2,1}(K)} 
         &\leq C(\|(\nabla^2 u) v_h\|_{L^1(K)}+\|\nabla u\nabla v_h \|_{L^1(K)} + \|u (\nabla v_h)\|_{L^1(K)} \\
         &~~~~~~~~~~+ \|(\nabla u  ) v_h\|_{L^1(K)} + \|u v_h \|_{L^1(K)} )\\
         &\leq C(\|u\|_{W^{2,\infty}(K)}\|v_h\|_{L^{1}(K)} + \|u\|_{W^{1,\infty}(K)}\|\nabla v_h\|_{L^1(K)})\\
         &\leq C\|u\|_{W^{2,\infty}(K)}\|v_h\|_{W^{1,1}(K)},
    \end{aligned} 
     \setlength{\belowdisplayskip}{4pt}   
    \end{equation}
    with the help of the Sobolev interpolation inequality. A substitution of \eqref{eq4.19} into \eqref{eq4.18} and summing over all triangular elements gives 
    \begin{equation*}
     \setlength{\abovedisplayskip}{4pt}   
         \|u v_h - \mathcal{I}_h(u v_h)\|_{L^1(\Omega)} \leq Ch^2\|u\|_{W^{2,\infty}(\Omega)}\|v_h\|_{W^{1,1}(\Omega)}\leq Ch^2\|u\|_{W^{2,\infty}(\Omega)}\|v_h\|_{H^1(\Omega)},
      \setlength{\belowdisplayskip}{4pt}   
    \end{equation*}
    which yields the desired conclusion.
\end{proof}
}

To obtain a consistency error similar to that in \eqref{m_hconsistency}, we see that 
\begin{align}
  \setlength{\abovedisplayskip}{4pt} 
    &\Big( \frac{\mathbf{m}^n_{e,h} - \mathbf{m}^{n-1}_{e,h}}{\Delta t},\mathbf{v}_h \Big)_h + \gamma(\nabla \mathbf{m}^n_{e,h},\nabla \mathbf{v}_h) \notag\\
    =& \beta(\nabla \mathbf{m}^{n}_{e,h},\nabla\mathcal{I}_h(\mathbf{v}_h \times \mathbf{m}^{n-1}_e)) + \gamma(\nabla \mathbf{m}^{n}_{e,h},\nabla\mathcal{I}_h((\mathbf{m}^{n-1}_{e}\cdot\mathbf{v}_h)\mathbf{m}^{n-1}_e))+\mathcal{E}^n_1(\mathbf{m}_e,\mathbf{v}_h)\notag\\
    &\left.
    \begin{aligned}
        &-\beta(\Delta \mathbf{m}^n_e,\mathbf{v}_h\times \mathbf{m}^{n-1}_e-\mathcal{I}_h(\mathbf{v}_h\times \mathbf{m}^{n-1}_e)) \\
        &-\gamma(\Delta \mathbf{m}^{n}_e,(\mathbf{m}^{n-1}_{e}\cdot\mathbf{v}_h)\mathbf{m}^{n-1}_e-\mathcal{I}_h((\mathbf{m}^{n-1}_{e}\cdot\mathbf{v}_h)\mathbf{m}^{n-1}_e))
    \end{aligned}
    \right\} \mathcal{E}^n_2(\mathbf{m}_e,\mathbf{v}_h) \notag\\
    &\left.
    \begin{aligned}
        &+ \beta(\nabla(\mathbf{m}^{n}_e - \mathbf{m}^{n}_{e,h}),\nabla\mathcal{I}_h(\mathbf{v}_h\times \mathbf{m}^{n-1}_e)) \\
        &+\gamma (\nabla(\mathbf{m}^n_e - \mathbf{m}^n_{e,h}),\nabla(\mathcal{I}_h((\mathbf{m}^{n-1}_{e}\cdot\mathbf{v}_h)\mathbf{m}^{n-1}_e) - \mathbf{v}_h))
    \end{aligned}
    \right\} \mathcal{E}^n_3(\mathbf{m}_e,\mathbf{v}_h)\label{consiseq}\\
    &\left.+(\frac{\mathbf{m}^n_{e,h} - \mathbf{m}^{n-1}_{e,h}}{\Delta t},\mathbf{v}_h)_h - (\frac{\mathbf{m}^n_{e,h} - \mathbf{m}^{n-1}_{e,h}}{\Delta t},\mathbf{v}_h)\right\}\mathcal{E}^n_4(\mathbf{m}_e,\mathbf{v}_h)\notag\\
    &\left.+(\frac{\mathbf{m}^n_{e,h} - \mathbf{m}^{n-1}_{e,h}}{\Delta t}-\partial_t\mathbf{m}^{n}_e,\mathbf{v}_h)\right\}\mathcal{E}^n_5(\mathbf{m}_e,\mathbf{v}_h) . \notag
    \setlength{\belowdisplayskip}{4pt}
\end{align}
Meanwhile, the denoted terms could be bounded as follows: 
\begin{align}
  \setlength{\abovedisplayskip}{4pt} 
    & 
    |\mathcal{E}^n_2(\mathbf{m}_e,\mathbf{v}_h)| \leq \epsilon \|\mathbf{v}_h\|^2_{H^1} + C(\|\mathbf{m}^n_e\|_{W^{{2,\infty}}})h^4 ,  \quad \mbox{(by Lemma \ref{lem4.3})} ,  \label{E2} 
\\
   & 
    |\mathcal{E}^n_3(\mathbf{m}_e,\mathbf{v}_h)| \leq \epsilon \|\mathbf{v}_h\|^2_{H^1} + C(\|\mathbf{m}^n_e\|_{W^{3,\infty}})h^4 ,  \quad \mbox{(by Lemma \ref{superconvergence})} ,  \label{E3} 
\\
  & 
    |\mathcal{E}^n_4(\mathbf{m}_e,\mathbf{v}_h)| \leq \epsilon \|\mathbf{v}_h\|^2_{H^1} + C(\|\partial_t\mathbf{m}^n_e\|_{H^1})h^4 , \quad \mbox{(by Lemma \ref{lemlumped})} , \label{E4} 
\\
  & 
    |\mathcal{E}^n_5(\mathbf{m}_e,\mathbf{v}_h)| \leq \epsilon \|\mathbf{v}_h\|^2_{L^2} + C(\|\partial_{tt}\mathbf{m}^n_e\|_{L^2})\Delta t^2 , \quad \mbox{(by the Taylor expansion in time)} .  \label{E5} 
  \setlength{\belowdisplayskip}{4pt}
\end{align} 
A careful interpolation and node-wise finite difference approximation analysis, combined \eqref{consiseq} with \eqref{E2}-\eqref{E5} and \eqref{E1}, enables us to derive the following consistency estimate: 
\begin{equation}\label{mconsistency} 
    \aligned
    \setlength{\abovedisplayskip}{4pt}  
    &(\frac{\mathbf{m}_{e, h}^{n} - \mathbf{m}_{e, h}^{n-1}}{\Delta t},\mathbf{v}_h )_h + \gamma(\nabla \mathbf{m}_{e, h}^n, \nabla \mathbf{v}_h ) \\ 
     = &  \sum_{K \in \mathcal{K}_h} |K| \Big( \beta ( \nabla_h  \mathbf{m}_e^n \cdot 
  ( {\cal A}_h \mathbf{v}_h \times \nabla_h \mathbf{m}_e^{n-1} 
    + \nabla_h \mathbf{v}_h \times {\cal A}_h \mathbf{m}_e^{n-1} ) )_{K}  
    +\mathcal{E}^n(\mathbf{m}_e,\mathbf{v}_h) \\ 
     & 
     + \gamma \big(  \nabla_h \mathbf{m}_e^n \cdot  
   ( {\cal A}_h ( \mathbf{m}_e^{n-1} \cdot \mathbf{v}_h ) \nabla_h \mathbf{m}_e^{n-1} )_K   
    - ( {\cal A}_h ( \mathbf{m}_e^n - \mathbf{m}_e^{n-1} ) \cdot  
     \nabla_h \mathbf{m}_e^{n-1} )  \cdot  \nabla_h (\mathbf{m}_e^{n-1}\cdot \mathbf{v}_h ) \big)_{K} \Big) , 
    \setlength{\belowdisplayskip}{4pt}  
    \endaligned
\end{equation} 
for any vector test function $\mathbf{v}_h \in \mathbf{V}_h$. Moreover, 
$\mathcal{E}^n(\mathbf{m}_e,\mathbf{v}_h)$ could be bounded by  
 $$
   \Delta t \sum^{\ell}_{n=1}|\mathcal{E}^n(\mathbf{m}_e,\mathbf{v}_h)| \leq \epsilon\Delta t\sum^{\ell}_{n=1}\|\mathbf{v}^n_h\|^2_{H^1}+C(\|\mathbf{m} \|_{W^{1,\infty}(0,T;\mathbf{W}^{3,\infty}(\Omega))},\| \partial_{tt}\mathbf{m}\|_{L^\infty(0,T;\mathbf{L}^2(\Omega))})(h^4 + \Delta t^2 ).
 $$
Again, the node-wise interpolation values of $\mathbf{m}_e$ at the node points have been extensively utilized in the discrete integral expansions on the right hand side of~\eqref{mconsistency}. 

\subsection{Error estimate}
The numerical error functions, at both the intermediate and renormalization stages, are introduced as 
\begin{equation} \label{error function-1} 
    \setlength{\abovedisplayskip}{4pt}  
         \widetilde{\mathbf{e}}_h^n = \mathbf{m}_{e, h}^n - \widetilde{\mathbf{m}}_h^n , \quad 
         \mathbf{e}_h^n = \mathbf{m}_{e, h}^n - \mathbf{m}_h^n .  
    \setlength{\belowdisplayskip}{4pt}  
\end{equation}
Notice that the interpolation values of the numerical error functions at the node points will play an important role in the later analysis. A subtraction of \eqref{m_hconsistency} from \eqref{mconsistency} leads to an error evolutionary equation, in the variational form: 
\begin{equation}\label{merror} 
    \aligned
    \setlength{\abovedisplayskip}{4pt}  
    &(\frac{\widetilde{\mathbf{e}}_h^{n} - \mathbf{e}_h^{n-1}}{\Delta t},\mathbf{v}_h )_h + \gamma(\nabla \widetilde{\mathbf{e}}_h^n, \nabla \mathbf{v}_h ) \\ 
     = & \sum_{K\in \mathcal{K}} |K| \Big( \beta ( \nabla_h  \widetilde{\mathbf{e}}_h^n \cdot 
  ({\cal A}_h \mathbf{v}_h \times \nabla_h \mathbf{m}_h^{n-1} ) 
    +  \nabla_h  \mathbf{m}_e^n \cdot ( {\cal A}_h \mathbf{v}_h \times \nabla_h \mathbf{e}^{n-1}_h  ) )_{K}  \\ 
    &
    + \beta ( \nabla_h  \widetilde{\mathbf{e}}_h^n \cdot 
  ( \nabla_h \mathbf{v}_h \times {\cal A}_h \mathbf{m}^{n-1}_h ) 
    + \nabla_h  \mathbf{m}_e^n \cdot  ( \nabla_h \mathbf{v}_h \times {\cal A}_h \mathbf{e}^{n-1}_h ) )_{K} \\ 
     & 
     + \gamma \big( \nabla_h \mathbf{m}_e^n \cdot  
   ( {\cal A}_h ( \mathbf{e}^{n-1}_h \cdot \mathbf{v}_h ) \nabla_h \mathbf{m}_e^{n-1} )    
   + \nabla_h \widetilde{\mathbf{e}}^n \cdot  
   ( {\cal A}_h ( \mathbf{m}^{n-1}_h \cdot \mathbf{v}_h ) \nabla_h \mathbf{m}^{n-1}_h ) \big)_{K}  \\ 
     & 
     + \gamma \big( \nabla_h \mathbf{m}_e^n \cdot  
   ( {\cal A}_h ( \mathbf{m}^{n-1}_h \cdot \mathbf{v} ) \nabla_h \mathbf{e}^{n-1}_h )   
    - ( {\cal A}_h ( \mathbf{m}_e^n - \mathbf{m}_e^{n-1} ) \cdot  
     \nabla_h \mathbf{m}^{n-1}_h )  \cdot  \nabla_h (\mathbf{e}^{n-1}_h\cdot \mathbf{v}_h ) \big)_{K}  \\
     & 
     - \gamma \big(  ( {\cal A}_h ( \mathbf{m}_e^n - \mathbf{m}_e^{n-1} ) \cdot  
     \nabla_h \mathbf{e}^{n-1}_h )  \cdot  \nabla_h (\mathbf{m}_e^{n-1}\cdot \mathbf{v}_h ) \big)_{K} \\
    &
    - \gamma \big( ( {\cal A}_h ( \widetilde{\mathbf{e}}^n_h - \mathbf{e}^{n-1}_h ) \cdot  
     \nabla_h \mathbf{m}_h^{n-1} )  \cdot  \nabla_h (\mathbf{m}_h^{n-1}\cdot \mathbf{v}_h ) \big)_{K} \Big) + \mathcal{E}^n(\mathbf{m}_e,\mathbf{v}_h) . 
    \setlength{\belowdisplayskip}{4pt}  
    \endaligned
\end{equation}


    Before proceeding into the error estimate, we need a few precise estimates between \(\widetilde{\mathbf{e}}_h^n\) and \(\mathbf{e}_h^n\). The mass-lumped finite element inner product coincides with the node-wise inner product, and the corresponding proof follows similarly to Lemma 4.2 in \cite{li2026stability}. 
    \medskip
    \begin{lemma}\label{lem6}
        Assume that $|\mathbf{m}_h^n| = 1$ and $|\widetilde{\mathbf{m}}_h^n| \geq 1$, at the node points. The following inequality is valid: 
        \begin{equation}\label{e_l^2} 
              \setlength{\abovedisplayskip}{4pt}   
            (1)\, \|\mathbf{e}_h^n\|_{L^2_h}^2 + \| \widetilde{\mathbf{e}}_h^n - \mathbf{e}_h^n \|_{L^2_h}^2   
            \le \|\widetilde{\mathbf{e}}_h^n\|_{L^2_h}^2 
            \le 2 ( \|\mathbf{e}_h^n\|_{L^2_h}^2 + \| \widetilde{\mathbf{e}}_h^n - \mathbf{e}_h^n \|_{L^2_h}^2 ) .  
             \setlength{\belowdisplayskip}{4pt}  
        \end{equation}
Furthermore, under an $\| \cdot \|_{L^\infty}$ bound that $\|\widetilde{\mathbf{e}}_h^n\|_{L^\infty} \leq h^{\epsilon_0/4}$, we have
        \begin{equation}\label{e_lh1} 
         \setlength{\abovedisplayskip}{4pt}  
            (2) \,\|\nabla \mathbf{e}_h^n \|_{L^2} \leq M_0 (\|\widetilde{\mathbf{e}}_h^n\|_{L^2}  
            + \|\nabla \widetilde{\mathbf{e}}_h^n\|_{L^2} ), \quad 
            \mbox{$M_0$ is a constant only dependent on $\Omega$} .  
          \setlength{\belowdisplayskip}{4pt}  
        \end{equation}
    \end{lemma}

Now we present the main result in the following theorem for the fully discrete scheme \eqref{discretescheme}.
\medskip
\begin{theorem}\label{thm1error}
    Assume that $\mathbf{m}_0$ and the exact solution $\mathbf{m}_e$ to the LLG equation satisfy $\mathbf{m}_e \in W^{1,\infty}(0,\,T;\,\mathbf{W}^{3,\infty}(\Omega)) \bigcap H^2(0,\,T;\,\mathbf{L}^2(\Omega))$. For scheme \eqref{discretescheme1}-\eqref{discretescheme2}, for any $0 <\epsilon_0<1$, if there exists the positive constant $h_0 > 0$ such that
    \begin{equation}\label{CFLcondition}
        \begin{array}{ll}
         \setlength{\abovedisplayskip}{4pt}  
             h^2\lesssim \Delta t \lesssim h^{\epsilon_0},\,\,&\text{in the 2D case}, 
          \setlength{\belowdisplayskip}{4pt}  
        \end{array}
    \end{equation}
    and $h\leq h_0$, then there holds the optimal error estimate, for any $1 \le n \le N$: 
    \begin{equation} \label{convergence-0} 
     \setlength{\abovedisplayskip}{4pt}  
        \max_{1\leq n\leq N}(\|\mathbf{m}^n_h - \mathbf{m}_e^n\|_{L^2} + \|\widetilde{\mathbf{m}}^n_h - \mathbf{m}^n_e\|_{L^2}) \leq C_0(h^2 + \Delta t) ,  \quad 
        \mbox{ $C_0 > 0$ is independent of $h,~\Delta t$ and $h_0$} . 
     \setlength{\belowdisplayskip}{4pt}  
    \end{equation}
\end{theorem}

\begin{proof} 
Two index values $p_0$ and $q_0$ are introduced to facilitate the nonlinear error estimates: 
\begin{equation} 
\setlength{\abovedisplayskip}{4pt}
  2 < p_0 = \frac{2}{1 - \frac14 \epsilon_0} , \, \, \,  q_0 = 8 \epsilon_0^{-1} < +\infty , 
  \quad \mbox{so that}  \, \, \, \frac{1}{p_0} + \frac{1}{q_0} = \frac12 . 
\setlength{\belowdisplayskip}{4pt}
  \label{p-q-1} 
\end{equation} 
In fact, $p_0$ is slightly greater than 2, and $q_0 < +\infty$. Meanwhile, by the regularity assumption on the exact solution, together with a-priori estimates for the Lagrange interpolation operator, the following functional bound becomes available to $\mathbf{m}_{e,h}$: 
\begin{equation} 
\setlength{\abovedisplayskip}{4pt}
  \| \nabla \mathbf{m}_{e, h}^k \|_{L^{p_0}} , \, \,  \Big\| \frac{ \mathbf{m}_{e,h}^{k+1} - \mathbf{m}_{e, h}^k }{\Delta t} \Big\|_{L^\infty} \le C^* , \quad \forall k \ge 0 .  
\setlength{\belowdisplayskip}{4pt}
  \label{exact-inf-1} 
\end{equation} 
In addition, we make the following a-priori assumption for the numerical error function at the previous time step:
\begin{equation}\label{a priori-1} 
  \setlength{\abovedisplayskip}{4pt} 
\| \mathbf{e}_h^{n-1} \|_{L^2_h} , \, \| \widetilde{\mathbf{e}}_h^{n-1} - \mathbf{e}_h^{n-1} \|_{L^2_h}  
  \le \Delta t^{1 - \frac{\epsilon_0}{8}} + h^{2 - \frac{\epsilon_0}{4}} , \quad   
 \| \nabla \widetilde{\mathbf{e}}_h^{n-1} \|_{L^2}  \le \Delta t^{\frac12 - \frac{\epsilon_0}{8}} 
   + h^{1 - \frac{\epsilon_0}{4}} . 
   \setlength{\belowdisplayskip}{4pt}
\end{equation}
Such an a-priori assumption is valid at $n=0$, and it will be recovered by the optimal rate convergence analysis at the next time step, as will be proved later. In turn, a maximum norm bound for the numerical error function at the intermediate stage could be derived as 
\begin{equation}\label{a priori-2} 
  \setlength{\abovedisplayskip}{4pt}
\begin{aligned} 
  & 
  \| \widetilde{\mathbf{e}}_h^{n-1} \|_{L^2_h} \le (\| \mathbf{e}^{n-1}_h \|_{L^2_h} 
  + \| \widetilde{\mathbf{e}}^{n-1}_h - \mathbf{e}^{n-1}_h \|_{L^2_h}) 
  \le 2 (\Delta t^{1 - \frac{\epsilon_0}{8}} + h^{2 - \frac{\epsilon_0}{4}}) , 
\\
  &  
   \| \widetilde{\mathbf{e}}_h^{n-1} \|_{L^\infty} \le \frac{C ( \| \widetilde{\mathbf{e}}^{n-1} \|_{L^2} 
   + \| \nabla \widetilde{\mathbf{e}}_h^{n-1} \|_{L^2} ) }{h^\frac{\epsilon_0}{4}} \le C ( \Delta t^{\frac12 - \frac{\epsilon_0}{8}} h^{-\frac{\epsilon_0}{4}} + h^{2 - \frac{\epsilon_0}{2}} )   \le C h^{\frac{\epsilon_0}{8}} , 
\end{aligned} 
  \setlength{\belowdisplayskip}{4pt}
\end{equation} 
in which the scaling law that $\Delta t^{\frac12 - \frac{\epsilon_0}{8}} \le h^{\epsilon_0 ( \frac12 - \frac{\epsilon_0}{8} )} \le h^{\frac{3 \epsilon_0}{8}}$ has been applied, provided that $\epsilon_0$ is sufficiently small. 
Under such an error bound, an application of the preliminary estimate~\eqref{e_lh1} gives 
\begin{equation} 
  \setlength{\abovedisplayskip}{4pt} 
\begin{aligned} 
  \|\nabla \mathbf{e}_h^{n-1} \|_{L^2} \le & M_0 (\|\widetilde{\mathbf{e}}^{n-1} \|_{L^2_h} 
            + \|\nabla \widetilde{\mathbf{e}}_h^{n-1} \|_{L^2} ) 
            \le (M_0 +\frac12) ( \Delta t^{\frac12 - \frac{\epsilon_0}{8}}   
   + h^{1 - \frac{\epsilon_0}{4}} ) 
\\
   \le & 
     (M_0 +\frac12) ( h^{\frac{3 \epsilon_0}{8}}  
   + h^{1 - \frac{\epsilon_0}{4}} ) \le (M_0 + 1)  h^{\frac{3 \epsilon_0}{8}} . 
\end{aligned} 
  \setlength{\belowdisplayskip}{4pt} 
   \label{a priori-3}
\end{equation} 
Based on the fact that $\frac22 - \frac{2}{p_0} = \frac{\epsilon_0}{4}$, an application of inverse inequality implies that 
\begin{equation} 
  \setlength{\abovedisplayskip}{4pt}  
  \|\nabla \mathbf{e}_h^{n-1} \|_{L^{p_0}} \le C h^{-\frac{\epsilon_0}{4}} 
  \cdot \|\nabla \mathbf{e}_h^{n-1} \|
  \le C (M_0 + 1)  h^{3\epsilon_0 /8} \le 1/2 . 
  \setlength{\belowdisplayskip}{4pt} 
   \label{a priori-4}
\end{equation} 
Subsequently, 
an $L^\infty \cap W^{1, p_0}$ bound for the numerical solution is valid at the previous time step: 
\begin{equation} \label{a priori-5}
  \setlength{\abovedisplayskip}{4pt}  
\begin{aligned} 
  & 
    \|\widetilde{\mathbf{m}}_h^{n-1} \|_{L^\infty} \le \| \mathbf{m}_{e, h}^{n-1} \|_{L^\infty}  
    + \| \widetilde{\mathbf{e}}_h^{n-1} \|_{L^\infty} \le C^* + 1/2 = \tilde{C}_1 , 
\\
  & 
    \|\nabla \mathbf{m}_h^{n-1} \|_{L^{p_0}} \le \|\nabla \mathbf{m}_e^{n-1} \|_{L^{p_0}}  
    + \|\nabla \mathbf{e}_h^{n-1} \|_{L^{p_0}} \le C^* + 1/2 = \tilde{C}_1. 
\end{aligned} 
  \setlength{\belowdisplayskip}{4pt} 
\end{equation} 

By taking a test function $\mathbf{v}_h = \widetilde{\mathbf{e}}_h^n$ in \eqref{merror}, we obtain  
    \begin{align}
    \setlength{\abovedisplayskip}{4pt} 
        &\frac{1}{2 \Delta t} ( \|\widetilde{\mathbf{e}}^n_h\|_{L^2_h}^2 - \|\mathbf{e}^{n-1}_h\|_{L^2_h}^2 + \|\widetilde{\mathbf{e}}_h^n - \mathbf{e}_h^{n-1}\|^2_{L^2_h} ) + \gamma \|\nabla \widetilde{\mathbf{e}}_h^n\|^2_{L^2} \nonumber \\ 
     = & \sum_{K \in \mathcal{K}_h} |K| \Big( \beta ( \nabla_h  \widetilde{\mathbf{e}}_h^n \cdot 
  ( {\cal A}_h \widetilde{\mathbf{e}}_h^n \times \nabla_h \mathbf{m}_h^{n-1} ) 
    +  \nabla_h  \mathbf{m}_e^n \cdot ( {\cal A}_h \widetilde{\mathbf{e}}_h^n  
     \times \nabla_h \mathbf{e}_h^{n-1}  ) )_{K} 
    \nonumber \\ 
    &
    + \beta ( \nabla_h  \widetilde{\mathbf{e}}_h^n \cdot 
  ( \nabla_h \widetilde{\mathbf{e}}_h^n \times {\cal A}_h \mathbf{m}^{n-1}_h ) 
    + \nabla_h  \mathbf{m}_e^n \cdot  ( \nabla_h \widetilde{\mathbf{e}}_h^n \times
     {\cal A}_h \mathbf{e}^{n-1}_h ) )_{K} \nonumber \\ 
     & 
     + \gamma \big( \nabla_h \mathbf{m}_e^n \cdot  
   ( {\cal A}_h ( \mathbf{e}^{n-1}_h \cdot \widetilde{\mathbf{e}}_h^n ) \nabla_h \mathbf{m}_e^{n-1} )    
   + \nabla_h \widetilde{\mathbf{e}}_h^n \cdot  
   ( {\cal A}_h ( \mathbf{m}^{n-1}_h \cdot \widetilde{\mathbf{e}}_h^n ) \nabla_h \mathbf{m}^{n-1}_h ) \big)_{K}  \nonumber  \\ 
     & 
     + \gamma \big( \nabla_h \mathbf{m}_e^n \cdot  
   ( {\cal A}_h ( \mathbf{m}^{n-1}_h \cdot \widetilde{\mathbf{e}}_h^n ) \nabla_h \mathbf{e}^{n-1}_h )   
    - ( {\cal A}_h ( \mathbf{m}_e^n - \mathbf{m}_e^{n-1} ) \cdot  
     \nabla_h \mathbf{m}^{n-1}_h )  \cdot  
     \nabla_h (\mathbf{e}^{n-1}_h\cdot \widetilde{\mathbf{e}}^n_h ) \big)_{K} \nonumber \\
     & 
     - \gamma \big(  ( {\cal A}_h ( \mathbf{m}_e^n - \mathbf{m}_e^{n-1} ) \cdot  
     \nabla_h \mathbf{e}^{n-1}_h )  \cdot  
     \nabla_h (\mathbf{m}_e^{n-1} \cdot \widetilde{\mathbf{e}}_h^n   ) \big)_{K}  \nonumber \\
    &
    - \gamma \big( ( {\cal A}_h ( \widetilde{\mathbf{e}}_h^n - \mathbf{e}_h^{n-1} ) \cdot  
     \nabla_h \mathbf{m}^{n-1}_h )  \cdot  
     \nabla_h (\mathbf{m}^{n-1}_h\cdot \widetilde{\mathbf{e}}^n_h   ) \big)_{K} \Big)+\mathcal{E}^n(\mathbf{m}_e,\mathbf{v}_h) .  
     \setlength{\belowdisplayskip}{4pt} 
     \label{emerror} 
\end{align} 

The nonlinear error inner product terms, which have been expressed in terms of discrete summation, have to be analyzed in a careful way. A direct application of discrete H\"older inequality would face certain technical difficulties, since the triangular meshes may not be uniform. Instead, a conversion of these discrete summations into an appropriate continuous integral is needed to facilitate the theoretical analysis. We begin with the first term, namely $J_1 := \beta \sum_{K} | K|  ( \nabla_h  \widetilde{\mathbf{e}}^n_h \cdot ( {\cal A}_h \widetilde{\mathbf{e}}_h^n \times \nabla_h \mathbf{m}^{n-1}_h )  )_{K}$. It is observed that, the component-wise values of $\nabla_h  \widetilde{\mathbf{e}}_h^n$, $\nabla_h \mathbf{m}_h^{n-1}$, are exactly equal to those of $\nabla  \widetilde{\mathbf{e}}_h^n$ and $\nabla \mathbf{m}_h^{n-1}$, respectively, over each triangular mesh $K\in \mathcal{K}_h$, since a linear element is utilized. Meanwhile, the error function $\widetilde{\mathbf{e}}_h^n$ is not piecewise constant, while ${\cal A}_h \widetilde{\mathbf{e}}^n_h$ has component-wise constant values over each edge in the triangular mesh. To handle this discrete summation in a more convenient way, we introduce a piecewise-constant continuous function: 
\begin{equation} 
  \setlength{\abovedisplayskip}{4pt} 
  {\bf P}_h ( {\cal A}_h \mathbf{f} ) = \left( \begin{array}{l} 
  {\cal A}_x \mathbf{f}_{i+\frac12,j}  \\ 
  {\cal A}_y \mathbf{f}_{i,j+\frac12} 
  \end{array} \right) , \quad \mbox{over $K_{i,j}$} , \quad \mbox{for any vector grid function $\mathbf{f}$} . 
  \setlength{\belowdisplayskip}{4pt} 
  \label{continuous extension-1} 
\end{equation} 
Notice that the two component values are evaluated over the edges $(i,j) \to (i+1,j)$, and $(i,j) \to (i,j+1)$, and each component value corresponds to a two-dimensional vector. Of course, ${\bf P}_h ( {\cal A}_h \mathbf{f}) $ is not in the finite element functional space, since it is not continuous at the node points. Moreover, in terms of a comparison estimate between such an averaged function and its original finite element version, the following inequalities are straightforward: 
\begin{equation} 
  \setlength{\abovedisplayskip}{4pt} 
  \breve{C}_1 \| \mathbf{f}_h \|_{L^p} \le \| \mathbf{f}_h \|_{L_h^p} 
    \le \breve{C}_2 \| \mathbf{f}_h \|_{L^p}  , \quad 
  \| {\bf P}_h ( {\cal A}_h \mathbf{f}_h ) \|_{L^p}
    \le \breve{C}_2 \| \mathbf{f}_h \|_{L^p}  , \quad 1 \le p < + \infty,\,\,\forall\,\mathbf{f}\in \mathbf{V}_h .  
   \setlength{\belowdisplayskip}{4pt} 
   \label{continuous extension-2} 
\end{equation} 
On the other hand,  the following equality is observed, due to the fact that $\nabla  \widetilde{\mathbf{e}}_h^n$, $\nabla \mathbf{m}_h^{n-1}$ and ${\bf P}_h ( {\cal A}_h \widetilde{\mathbf{e}}^n )$ becomes piecewise-constant vector functions over each triangular mesh: 
\begin{equation} 
  \setlength{\abovedisplayskip}{4pt} 
  | K|  ( \nabla_h  \widetilde{\mathbf{e}}^n_h \cdot ( {\cal A}_h \widetilde{\mathbf{e}}_h^n \times \nabla_h \mathbf{m}^{n-1}_h )  )_{K} = \int_{K} \,   \nabla  \widetilde{\mathbf{e}}_h^n \cdot ( {\bf P}_h ( {\cal A}_h \widetilde{\mathbf{e}}_h^n ) \times \nabla \mathbf{m}_h^{n-1} )  \, d {\bf x} . 
  \setlength{\belowdisplayskip}{4pt} 
  \label{continuous extension-3} 
\end{equation} 

As a result, a bound for the first nonlinear inner product term on the right hand side of~\eqref{emerror} could be derived as follows, with the help of H\"older inequality (at the continuous level): 
\begin{align}  
  \setlength{\abovedisplayskip}{4pt}  
  J_1 = & \beta \sum_{K\in \mathcal{K}_h} |K|  ( \nabla_h  \widetilde{\mathbf{e}}_h^n \cdot ( {\cal A}_h \widetilde{\mathbf{e}}_h^n \times \nabla_h \mathbf{m}_h^{n-1} )  )_{K} 
  =  ( \nabla  \widetilde{\mathbf{e}}_h^n , {\bf P}_h ( {\cal A}_h \widetilde{\mathbf{e}}_h^n ) \times \nabla \mathbf{m}_h^{n-1} ) \notag
  \\
   \le  & 
   \beta \| \nabla \widetilde{\mathbf{e}}_h^n \|_{L^2}  
  \cdot \| {\bf P}_h ( {\cal A}_h \widetilde{\mathbf{e}}_h^n ) \|_{L^{q_0}} 
  \cdot \| \nabla \mathbf{m}_h^{n-1} \|_{L^{p_0}}  
   \le \beta \breve{C}_2 \| \nabla \widetilde{\mathbf{e}}_h^n \|_{L^2}  
  \cdot \| \widetilde{\mathbf{e}}_h^n \|_{L^{q_0}} \cdot \| \nabla \mathbf{m}_h^{n-1} \|_{L^{p_0}}  \label{convergence-NLE-1} 
\\
  \le & 
  \beta \breve{C}_2 \tilde{C}_1 \| \nabla \widetilde{\mathbf{e}}_h^n \|_{L^2}  
  \cdot \| \widetilde{\mathbf{e}}_h^n \|_{L^{q_0}} ,  \quad \mbox{(by the a-priori bound~\eqref{a priori-5})}. \notag
\end{align} 
  
The second nonlinear inner product term on the right hand side of~\eqref{emerror} could be similarly bounded: 
\begin{equation} 
\begin{aligned}  
  \setlength{\abovedisplayskip}{4pt}  
  J_2 = & \beta \sum_{K \in \mathcal{K}_h} | K |  ( \nabla_h  \mathbf{m}_e^n \cdot ( {\cal A}_h \widetilde{\mathbf{e}}_h^n  
     \times \nabla_h \mathbf{e}^{n-1}_h  ) )_{K}
  =  ( \nabla  \mathbf{m}_{e, h}^n , {\bf P}_h ( {\cal A}_h \widetilde{\mathbf{e}}_h^n ) 
     \times \nabla \mathbf{e}_h^{n-1}  )
  \\
   \le  & 
      \beta \breve{C}_2 \| \nabla  \mathbf{m}_{e, h}^n \|_{L^{p_0}} 
   \cdot  \| \widetilde{\mathbf{e}}_h^n \|_{L^{q_0}} \cdot \| \nabla \mathbf{e}_h^{n-1} \|_{L^2} 
  \le  \beta \breve{C}_2 C^* \| \nabla \mathbf{e}_h^{n-1} \|_{L^2}  
  \cdot \| \widetilde{\mathbf{e}}_h^n \|_{L^{q_0}}  . 
\end{aligned} 
  \setlength{\belowdisplayskip}{4pt} 
  \label{convergence-NLE-2} 
\end{equation} 

The third nonlinear inner product term on the right hand side of~\eqref{emerror} vanishes, due to the fact that the vectors $\nabla_h \widetilde{\mathbf{e}}^n_h \times {\cal A}_h \mathbf{m}^{n-1}_h$ and $\nabla_h \widetilde{\mathbf{e}}_h^n$ are orthogonal at each wedge section: 
\begin{equation} 
  \setlength{\abovedisplayskip}{4pt}  
  J_3 = \beta \sum_{K \in \mathcal{K}_h} | K | ( \nabla_h  \widetilde{\mathbf{e}}_h^n \cdot 
  ( \nabla_h \widetilde{\mathbf{e}}_h^n \times {\cal A}_h \mathbf{m}^{n-1}_h ) )_{K} = 0 . 
  \setlength{\belowdisplayskip}{4pt}  
  \label{convergence-NLE-3} 
\end{equation} 
The rest nonlinear inner product terms on the right hand side of~\eqref{emerror} could be similarly analyzed: 
\begin{align} 
  \setlength{\abovedisplayskip}{4pt} 
  J_4 = & \beta \sum_{K \in \mathcal{K}_h} | K| ( \nabla_h  \mathbf{m}_e^n \cdot  ( 
  \nabla_h \widetilde{\mathbf{e}}_h^n \times   {\cal A}_h \mathbf{e}_h^{n-1} ) )_{K}   
    = \beta ( \nabla  \mathbf{m}_{e, h}^n ,   \nabla \widetilde{\mathbf{e}}_h^n \times
    ( {\bf P}_h ( {\cal A}_h \mathbf{e}^{n-1} ) ) )  \nonumber 
\\
  \le & 
  \beta \| \nabla  \mathbf{m}_{e, h}^n \|_{L^{p_0}} \cdot  \| \nabla \widetilde{\mathbf{e}}_h^n \|_{L^2} 
    \cdot \| {\bf P}_h ( {\cal A}_h \mathbf{e}^{n-1} ) \|_{L^{q_0}} 
  \le \beta C^* \breve{C}_2   \| \nabla \widetilde{\mathbf{e}}_h^n \|_{L^2} \cdot \| \mathbf{e}_h^{n-1}  \|_{L^{q_0}} , 
  \label{convergence-NLE-4} 
\\
  J_5 = & 
      \gamma \sum_{K \in \mathcal{K}_h} | K | \big( \nabla_h \mathbf{m}_e^n \cdot  
   ( {\cal A}_h ( \mathbf{e}_h^{n-1} \cdot \widetilde{\mathbf{e}}_h^n ) \nabla_h \mathbf{m}_e^{n-1} )  \big)_{K}  
     = \gamma \big( \nabla \mathbf{m}_{e, h}^n ,  
   {\bf P}_h ( {\cal A}_h ( \mathbf{e}_h^{n-1} \cdot \widetilde{\mathbf{e}}_h^n ) ) \nabla \mathbf{m}_{e, h}^{n-1} \big)  
   \nonumber  \\ 
     \le & 
      \gamma \| \nabla \mathbf{m}_{e, h}^n \|_{L^{p_0}}   \cdot 
      \| {\bf P}_h ( {\cal A}_h ( \mathbf{e}^{n-1}_h \cdot \widetilde{\mathbf{e}}_h^n ) ) \|_{L^{q_0/2}} 
        \cdot \| \nabla \mathbf{m}_{e, h}^{n-1} \|_{L^{p_0}} 
      \le \gamma ( C^* )^2 \breve{C}_2 \| \mathbf{e}_h^{n-1} \|_{q_0} \cdot \| \widetilde{\mathbf{e}}_h^n \|_{q_0} , 
     \label{convergence-NLE-5}   \\ 
    J_6 = & 
     \gamma \sum_{K \in \mathcal{K}_h} |K| \big( \nabla_h \widetilde{\mathbf{e}}_h^n \cdot  
   ( {\cal A}_h ( \mathbf{m}^{n-1}_h \cdot \widetilde{\mathbf{e}}_h^n ) \nabla_h \mathbf{m}^{n-1}_h ) \big)_{K}  
   = \gamma \big( \nabla \widetilde{\mathbf{e}}_h^n ,   
   {\bf P}_h ( {\cal A}_h ( \mathbf{m}_h^{n-1} \cdot \widetilde{\mathbf{e}}_h^n ) ) 
   \nabla \mathbf{m}_h^{n-1} )   \nonumber  \\ 
   \le & \gamma \| \nabla \widetilde{\mathbf{e}}_h^n \|_{L^2} \cdot    
   \| {\bf P}_h ( {\cal A}_h ( \mathbf{m}^{n-1}_h \cdot \widetilde{\mathbf{e}}_h^n ) ) \|_{L^{q_0}} 
   \cdot \| \nabla \mathbf{m}_h^{n-1} \|_{L^{p_0}} \nonumber \\ 
  \le & \gamma \breve{C}_2 \| \nabla \widetilde{\mathbf{e}}_h^n \|_{L^2} \cdot    
   \| \mathbf{m}_h^{n-1} \|_{L^\infty} \cdot \| \widetilde{\mathbf{e}}_h^n \|_{L^{q_0}} 
   \cdot \| \nabla \mathbf{m}_h^{n-1} \|_{L^{p_0}}  
   \le \gamma \tilde{C}_1^2 \breve{C}_2 \| \nabla \widetilde{\mathbf{e}}_h^n \|_{L^2} 
    \cdot \| \widetilde{\mathbf{e}}_h^n \|_{L^{q_0}} ,  \label{convergence-NLE-6} \\  
     J_7 = & 
      \gamma \sum_{K \in \mathcal{K}_h} |K| \big( \nabla_h \mathbf{m}_e^n \cdot  
   ( {\cal A}_h ( \mathbf{m}^{n-1}_h \cdot \widetilde{\mathbf{e}}_h^n ) \nabla_h \mathbf{e}^{n-1}_h ) \big)_{K} 
   = \gamma \big( \nabla \mathbf{m}_{e, h}^n ,   
   {\bf P}_h ( {\cal A}_h ( \mathbf{m}^{n-1}_h \cdot \widetilde{\mathbf{e}}_h^n ) )
     \nabla \mathbf{e}_h^{n-1} \big) \nonumber \\ 
     \le & 
     \gamma \| \nabla \mathbf{m}_{e, h}^n \|_{L^{p_0}} \cdot    
   \| {\bf P}_h ( {\cal A}_h ( \mathbf{m}^{n-1}_h \cdot \widetilde{\mathbf{e}}_h^n ) ) \|_{L^{q_0}} 
     \cdot \| \nabla \mathbf{e}_h^{n-1} \|_{L^2}   \nonumber \\
     \le & 
     \gamma \breve{C}_2 \| \nabla \mathbf{m}_{e, h}^n \|_{L^{p_0}} \cdot    
   \| \mathbf{m}^{n-1} \|_{L^\infty} \cdot \| \widetilde{\mathbf{e}}^n_h \|_{q_0} 
     \cdot \| \nabla \mathbf{e}_h^{n-1} \|_{L^2}  
     \le \gamma C^* \tilde{C}_1 \breve{C}_2 \| \widetilde{\mathbf{e}}_h^n \|_{L^{q_0}} 
     \cdot \| \nabla \mathbf{e}_h^{n-1} \|_{L^2} ,  \label{convergence-NLE-7} \\ 
     J_8 = & 
       - \gamma \sum_{K \in \mathcal{K}_h} | K| \big( 
     ( {\cal A}_h ( \mathbf{m}_e^n - \mathbf{m}_e^{n-1} ) \cdot  
     \nabla_h \mathbf{m}_h^{n-1} )  \cdot  
     \nabla_h (\mathbf{e}^{n-1}_h\cdot \widetilde{\mathbf{e}}_h^n ) \big)_{K}  \nonumber \\ 
     = & 
     - \gamma \big( 
     {\bf P}_h ( {\cal A}_h ( \mathbf{m}_e^n - \mathbf{m}_e^{n-1} ) ) \cdot  
     \nabla \mathbf{m}_h^{n-1}  ,   
     \nabla \mathbf{e}_h^{n-1} \cdot {\bf P}_h ({\cal A}_h \widetilde{\mathbf{e}}_h^n ) 
     + {\bf P}_h ( {\cal A}_h  \mathbf{e}_h^{n-1} ) \cdot \nabla \widetilde{\mathbf{e}}_h^n \big)  
      \nonumber \\
      \le & 
      \gamma \breve{C}_2 \| \mathbf{m}_{e, h}^n - \mathbf{m}_{e, h}^{n-1} \|_{L^\infty} \cdot  
      \| \nabla \mathbf{m}_h^{n-1} \|_{L^{p_0}}    
      \big( \| \nabla \mathbf{e}_h^{n-1} \|_{L^2} \cdot \| \widetilde{\mathbf{e}}_h^n \|_{L^{q_0}}  
     + \|  \mathbf{e}_h^{n-1} \|_{L^{q_0}} \cdot \| \nabla \widetilde{\mathbf{e}}_h^n \|_{L^2} \big)  
      \nonumber \\
      \le & 
      \gamma C^* \tilde{C}_1 \breve{C}_2 \Delta t 
      \big( \| \nabla \mathbf{e}_h^{n-1} \|_{L^2} \cdot \| \widetilde{\mathbf{e}}_h^n \|_{L^{q_0}}  
     + \|  \mathbf{e}_h^{n-1} \|_{L^{q_0}} \cdot \| \nabla \widetilde{\mathbf{e}}_h^n \|_{L^2} \big) , 
      \label{convergence-NLE-8} \\
     J_9 = & 
     - \gamma \sum_{K \in \mathcal{K}_h} | K| \big(  ( {\cal A}_h ( \mathbf{m}_e^n - \mathbf{m}_e^{n-1} ) \cdot  
     \nabla_h \mathbf{e}^{n-1}_h )  \cdot  
     \nabla_h (\mathbf{m}_e^{n-1} \cdot \widetilde{\mathbf{e}}_h^n   ) \big)_{K}  \nonumber \\ 
    = & 
    - \gamma \big(  {\bf P}_h ( {\cal A}_h ( \mathbf{m}_e^n - \mathbf{m}_e^{n-1} ) ) \cdot  
     \nabla \mathbf{e}_h^{n-1} ,   
     \nabla \mathbf{m}_{e, h}^{n-1} \cdot {\cal P}_h ( {\cal A}_h \widetilde{\mathbf{e}}_h^n ) 
     + {\bf P}_h ({\cal A}_h (\mathbf{m}_e^{n-1} ) \cdot \nabla \widetilde{\mathbf{e}}_h^n ) \big)  \nonumber \\ 
    \le & 
      \gamma \breve{C}_2  \| \mathbf{m}_{e, h}^n - \mathbf{m}_{e, h}^{n-1} \|_{L^\infty} \cdot  
     \| \nabla \mathbf{e}_h^{n-1} \|_{L^2}    
    \big( \nabla \mathbf{m}_{e, h}^{n-1} \|_{L^{p_0}} \cdot \| \widetilde{\mathbf{e}}_h^n \|_{L^{q_0}}  
     + \| \mathbf{m}_e^{n-1} \|_{L^\infty} \cdot \| \nabla \widetilde{\mathbf{e}}_h^n \|_{L^2} \big)  \nonumber \\ 
   \le & 
      \gamma (C^*)^2 \breve{C}_2  \Delta t   
     \| \nabla \mathbf{e}_h^{n-1} \|_{L^2}    ( \| \widetilde{\mathbf{e}}_h^n \|_{L^{q_0}} 
     + \| \nabla \widetilde{\mathbf{e}}_h^n \|_{L^2})  , \label{convergence-NLE-9}  \\ 
    J_{10} = &
    - \gamma \sum_{K\in \mathcal{K}_h} | K| \big( ( {\cal A}_h ( \widetilde{\mathbf{e}}_h^n - \mathbf{e}^{n-1}_h ) \cdot  
     \nabla_h \mathbf{m}^{n-1}_h )  \cdot  
     \nabla_h (\mathbf{m}^{n-1}_h\cdot \widetilde{\mathbf{e}}_h^n) \big)_{K}  \nonumber \\ 
     = &
    - \gamma \big( {\bf P}_h ( {\cal A}_h ( \widetilde{\mathbf{e}}_h^n - \mathbf{e}_h^{n-1} ) ) \cdot  
     \nabla \mathbf{m}_h^{n-1} ,   
     \nabla \mathbf{m}_h^{n-1} \cdot {\bf P}_h ( {\cal A}_h \widetilde{\mathbf{e}}_h^n ) 
    + {\bf P}_h ( {\cal A}_h \mathbf{m}^{n-1}_h ) \cdot \nabla \widetilde{\mathbf{e}}_h^n  \big)  \nonumber \\ 
    \le &
      \gamma \breve{C}_2^2 ( \| \widetilde{\mathbf{e}}_h^n \|_{L^{q_0}} + \| \mathbf{e}_h^{n-1} \|_{L^{q_0}}  )   
     \| \nabla \mathbf{m}_h^{n-1} \|_{L^{p_0}}     
     \big( \| \nabla \mathbf{m}_h^{n-1} \|_{L^{p_0}} \cdot \| \widetilde{\mathbf{e}}_h^n \|_{L^{q_0}}  
    + \| \mathbf{m}_h^{n-1} \|_{L^\infty} \cdot \| \nabla \widetilde{\mathbf{e}}_h^n  \|_{L^2} \big)  \nonumber \\ 
   \le &
      \gamma \tilde{C}_1^2 \breve{C}_2^2 ( \| \widetilde{\mathbf{e}}_h^n \|_{L^{q_0}} 
      + \| \mathbf{e}_h^{n-1} \|_{L^{q_0}}  )     
     ( \| \widetilde{\mathbf{e}}_h^n \|_{L^{q_0}}  + \| \nabla \widetilde{\mathbf{e}}_h^n  \|_{L^2} )  . 
      \label{convergence-NLE-10}
     \setlength{\belowdisplayskip}{4pt} 
\end{align}

Subsequently, a substitution of~\eqref{convergence-NLE-1}-\eqref{convergence-NLE-10} 
into \eqref{emerror} leads to 
\begin{equation} 
  \setlength{\abovedisplayskip}{4pt}    
  \begin{aligned} 
    & 
    \frac{1}{2 \Delta t} (\|\widetilde{\mathbf{e}}_h^n\|_{L^2_h}^2 - \|\mathbf{e}_h^{n-1}\|_{L^2_h}^2 + \|\widetilde{\mathbf{e}}_h^n - \mathbf{e}^{n-1}_h\|^2_{L^2_h}) + \gamma \|\nabla \widetilde{\mathbf{e}}_h^n\|^2_{L^2} 
    - \frac{\gamma}{16 M_0} \| \nabla \widetilde{\mathbf{e}}_h^n \|_{L^2}   
  \cdot \| \nabla \mathbf{e}_h^{n-1} \|_{L^2}
\\
  \le & 
 (\tilde{C}_2 + \frac{\tilde{C}_3}{2})  \| \mathbf{e}_h^{n-1} \|_{L^{q_0}}^2 + ( \tilde{C}_2  + \frac{\breve{C}_2^2}{2} + \frac{3\tilde{C}_3}{2}) \| \widetilde{\mathbf{e}}_h^n \|_{L^{q_0}}^2  
  + (\tilde{C}_4 +1)\| \nabla \mathbf{e}_h^{n-1} \|_{L^2}\| \widetilde{\mathbf{e}}_h^n \|_{L^{q_0}}  
\\
  & 
  + \| \nabla \widetilde{\mathbf{e}}_h^n \|_{L^2} (\tilde{C}_5 \| \widetilde{\mathbf{e}}_h^n \|_{L^{q_0}}  
  + ( 1 + \tilde{C}_3 +  \beta C^* \breve{C}_2) \| \mathbf{e}_h^{n-1} \|_{L^{q_0}} ) + |\mathcal{E}^n(\mathbf{m}_e,\mathbf{v}_h)|, 
  \setlength{\belowdisplayskip}{4pt}  
\end{aligned} 
    \label{convergence-3-2} 
\end{equation} 
with $\tilde{C}_2 = \frac{\gamma (C^*)^2 \breve{C}_2}{2}$, $\tilde{C}_3 = \gamma \tilde{C}_1^2 \breve{C}_2^2$, $\tilde{C}_4 = ( \beta + \gamma \tilde{C}_1 ) C^* \breve{C}_2 + 1$, $\tilde{C}_5 = (\beta + \gamma \tilde{C}_1 ) \tilde{C}_1 \breve{C}_2 + 1 + \tilde{C}_3$. On the other hand, a repeated application of Cauchy inequality gives 
\begin{align} 
  \setlength{\abovedisplayskip}{4pt}  
    & 
   \frac{\gamma}{16 M_0} \| \nabla \widetilde{\mathbf{e}}_h^n \|_{L^2}  
  \cdot \| \nabla \mathbf{e}_h^{n-1} \|_{L^2}  
  \le \frac{\gamma}{16} \| \nabla \widetilde{\mathbf{e}}_h^n \|_{L^2}^2  
  + \frac{\gamma}{64 M_0^2} \| \nabla \mathbf{e}_h^{n-1} \|_{L^2}^2 , 
    \label{convergence-3-3} 
\\
  & 
  ( \tilde{C}_4 + 1)\| \nabla \mathbf{e}_h^{n-1} \|_{L^2}  \| \widetilde{\mathbf{e}}_h^n \|_{L^{q_0}}  
  \le  
  \frac{\gamma}{64 M_0^2} \| \nabla \mathbf{e}_h^{n-1} \|_{L^2}^2 
  + 16 M_0^2 \gamma^{-1} ( \tilde{C}_4 + 1)^2\| \widetilde{\mathbf{e}}_h^n \|^2_{L^{q_0}} , 
   \label{convergence-3-4} 
\\
  & 
  \| \nabla \widetilde{\mathbf{e}}_h^n \|_{L^2} (\tilde{C}_5 \| \widetilde{\mathbf{e}}_h^n \|_{L^{q_0}}  
  + ( 1 + \tilde{C}_3 + \beta C^* \breve{C}_2) \| \mathbf{e}_h^{n-1} \|_{L^{q_0}} )  \notag 
\\
  \le & 
  \frac{\gamma}{16} \| \nabla \widetilde{\mathbf{e}}_h^n \|_{L^2}^2 
  + 8 \gamma^{-1}  (\tilde{C}_5^2 \| \widetilde{\mathbf{e}}_h^n \|_{L^{q_0}}^2  
  + ( 1 + \tilde{C}_3 + \beta C^* \breve{C}_2)^2 \| \mathbf{e}_h^{n-1} \|_{L^{q_0}}^2 )  .  
  \label{convergence-3-5} 
  \setlength{\belowdisplayskip}{4pt}  
\end{align} 
Therefore, a substitution of~\eqref{convergence-3-3}-\eqref{convergence-3-5} into \eqref{convergence-3-2} yields 
\begin{equation} 
  \setlength{\abovedisplayskip}{4pt}    
  \begin{aligned} 
    & 
    \frac{1}{2 \Delta t} ( \|\widetilde{\mathbf{e}}_h^n\|_{L^2_h}^2 - \|\mathbf{e}_h^{n-1}\|_{L^2_h}^2 + \|\widetilde{\mathbf{e}}_h^n - \mathbf{e}_h^{n-1}\|_{L^2_h}^2 ) + \frac{7 \gamma}{8} \|\nabla \widetilde{\mathbf{e}}_h^n \|_{L^2}^2  
    - \frac{\gamma}{32 M_0^2} \|\nabla \mathbf{e}_h^{n-1} \|_{L^2}^2 
\\
  \le & 
  \tilde{C}_6  \| \mathbf{e}_h^{n-1} \|_{L^{q_0}}^2 + \tilde{C}_7 \| \widetilde{\mathbf{e}}_h^n \|_{L^{q_0}}^2  + |\mathcal{E}^n(\mathbf{m}_e,\mathbf{v}_h)| , 
  \setlength{\belowdisplayskip}{4pt}  
\end{aligned} 
    \label{convergence-3-6} 
\end{equation}   
with $\tilde{C}_6 = \tilde{C}_2 +\frac{\widetilde{C}_3}{2} + 8 \gamma^{-1}  ( 1 + \tilde{C}_3 + \beta C^* \breve{C}_2)^2$, $\tilde{C}_7 = \tilde{C}_2 + \frac{\breve{C}_2^2}{2} +\frac{3\widetilde{C}_3}{2}+ 16 M_0^2 \gamma^{-1} (\tilde{C}_4+ 1)^2+8\gamma^{-1}\widetilde{C}_5^2$. Meanwhile, with $2 < q_0 = 8 \epsilon_0^{-1} < + \infty$, an application of the interpolation inequality~\eqref{interpolation ineq-1} 
indicates that 
\begin{equation} 
    \setlength{\abovedisplayskip}{4pt}  
\begin{aligned} 
  & 
  \|\boldsymbol{f}_h \|_{L^{q_0}} \le  \breve{C}_0 \|\boldsymbol{f}_h \|_{L^2}^{\frac{\epsilon_0}{4}} \cdot 
    ( \| \boldsymbol{f}_h \|_{L^2} + \| \nabla \boldsymbol{f}_h \|_{L^2} )^{1 - \frac{\epsilon_0}{4} }    
   \le \breve{C}_0\breve{C}_1 ( \| \boldsymbol{f}_h \|_{L^2} 
   + \| \boldsymbol{f}_h \|_{L^2}^{\frac{\epsilon_0}{4}} \cdot 
    \| \nabla \boldsymbol{f}_h \|_{L^2}^{1 - \frac{\epsilon_0}{4} }), 
\end{aligned}  
     \setlength{\belowdisplayskip}{4pt}   
    \label{convergence-3-7} 
\end{equation}
so that $\|\boldsymbol{f}_h \|_{L^{q_0}}^2  
    \le C ( \epsilon_0, \alpha )  \| \boldsymbol{f}_h \|_{L^2}^2  
    + \alpha \| \nabla \boldsymbol{f}_h \|_{L^2}^2,\,\,\forall \alpha > 0$, in which the Young's inequality has been applied in the last step. As a consequence, we obtain  
\begin{align} 
  & 
  \tilde{C}_6  \| \widetilde{\mathbf{e}}_h^n \|_{L^{q_0}}^2  
  \le \tilde{C}_{8}  (\epsilon_0, \gamma) \| \widetilde{\mathbf{e}}_h^n \|_{L^2}^2 
  + \frac{\gamma}{8} \| \nabla \widetilde{\mathbf{e}}_h^n \|_{L^2}^2 , 
  \label{convergence-3-8} 
\\
  &
   \tilde{C}_7  \| \mathbf{e}^{n-1}_h \|_{L^{q_0}}^2  
  \le \tilde{C}_{9}  (\epsilon_0, \gamma, M_0) \| \mathbf{e}^{n-1}_h \|_{L^2}^2 
  + \frac{\gamma}{32 M_0^2} \| \nabla \mathbf{e}_h^{n-1} \|_{L^2}^2, \label{convergence-3-9} 
\end{align}  
and its substitution into~\eqref{convergence-3-6} leads to 
\begin{equation} 
  \setlength{\abovedisplayskip}{4pt}    
\begin{aligned} 
    & 
    \frac{1}{2 \Delta t} ( \|\widetilde{\mathbf{e}}_h^n\|_{L^2_h}^2 - \|\mathbf{e}_h^{n-1}\|_{L^2_h}^2 + \|\widetilde{\mathbf{e}}_h^n - \mathbf{e}_h^{n-1}\|^2_{L_h^2} ) + \frac{3 \gamma}{4} \|\nabla \widetilde{\mathbf{e}}_h^n \|_{L^2}^2 
    - \frac{\gamma}{16 M_0^2} \|\nabla \mathbf{e}_h^{n-1} \|_{L^2}^2   
\\
  \le & 
  C(\| \mathbf{e}^{n-1} \|_{L^2}^2 + \| \widetilde{\mathbf{e}}^n \|_{L^2}^2)   
  + |\mathcal{E}^n(\mathbf{m}_e,\mathbf{v}_h)|. 
\end{aligned} 
   \label{convergence-3-10} 
  \setlength{\belowdisplayskip}{4pt}  
\end{equation}    
Meanwhile, by the a-priori error estimate~\eqref{e_lh1}, we observe that $\|\nabla \mathbf{e}_h^{n-1} \|^2 \le 2 M_0^2  (\| \widetilde{\mathbf{e}}_h^{n-1} \|^2 +  \| \nabla \widetilde{\mathbf{e}}_h^{n-1} \|^2)$, and arrive at 
\begin{align} 
    & 
    \frac{1}{2 \Delta t} ( \|\widetilde{\mathbf{e}}_h^n\|_{L^2_h}^2 - \|\mathbf{e}_h^{n-1}\|_{L^2_h}^2 + \|\widetilde{\mathbf{e}}_h^n - \mathbf{e}_h^{n-1}\|^2_{L^2_h} ) + \frac{3 \gamma}{4} \|\nabla \widetilde{\mathbf{e}}_h^n \|^2 
    - \frac{\gamma}{8} \|\nabla \widetilde{\mathbf{e}}_h^{n-1} \|^2  \notag
\\
  \le & 
  C(\| \mathbf{e}^{n-1}_h \|_{L^2}^2 + \| \widetilde{\mathbf{e}}_h^n \|_{L^2}^2)  + |\mathcal{E}^n(\mathbf{m}_e,\mathbf{v}_h)|. \label{convergence-3-11} 
\end{align}
In addition, by the $\ell^2$ error estimate~\eqref{e_l^2} in the normalization stage, we get    
\begin{align} 
    & 
    \frac{1}{2 \Delta t} ( \| \mathbf{e}_h^n\|_{L^2_h}^2 - \|\mathbf{e}_h^{n-1}\|_{L^2_h}^2 + \|\widetilde{\mathbf{e}}_h^n - \mathbf{e}_h^{n-1}\|^2_{L^2_h} + \| \widetilde{\mathbf{e}}_h^n - \mathbf{e}_h^n \|_{L^2_h}^2) + \frac{3 \gamma}{4} \|\nabla \widetilde{\mathbf{e}}_h^n \|_{L^2}^2  
    - \frac{\gamma}{8} \|\nabla \widetilde{\mathbf{e}}_h^{n-1} \|_{L^2}^2 \notag 
\\
  \le& C(\| \mathbf{e}^{n-1}_h \|_{L^2_h}^2 + \|\mathbf{e}^{n}_h\|_{L^2_h}^2 + \| \widetilde{\mathbf{e}}_h^n -\mathbf{e}^{n}_h\|_{L^2_h}^2)  + |\mathcal{E}^n(\mathbf{m}_e,\mathbf{v}_h)|,\label{convergence-3-12}  
\end{align} 
in which the comparison estimate~\eqref{continuous extension-2} (between the discrete and continuous norms) has been applied in the second step. A summation in time implies that 
\begin{equation} 
  \setlength{\abovedisplayskip}{4pt}    
    \| \mathbf{e}_h^n\|_{L^2_h}^2 
    + \frac12 \sum_{k=1}^n \| \widetilde{\mathbf{e}}_h^k - \mathbf{e}_h^k \|_{L^2_h}^2  
    + \frac{5 \gamma \Delta t}{4} \sum_{k=1}^n \|\nabla \widetilde{\mathbf{e}}_h^k \|_{L^2}^2  
  \le  
   C  \Delta t \sum_{k=0}^n \| \mathbf{e}_h^k \|_{L^2_h}^2  
  +  \Delta t \sum_{k=1}^n |\mathcal{E}^n(\mathbf{m}_e,\mathbf{v}_h)|.  
  \setlength{\belowdisplayskip}{4pt}  
    \label{convergence-3-13} 
\end{equation}   
Therefore, an application of discrete Gronwall inequality yields the desired error estimate 
\begin{equation} 
  \setlength{\abovedisplayskip}{4pt}    
    \| \mathbf{e}_h^n\|_{L^2_h}  
    + \Big( \frac12 \sum_{k=0}^n \| \widetilde{\mathbf{e}}_h^k - \mathbf{e}_h^k \|_{L^2_h}^2  \Big)^\frac12 
    + \Big( \gamma \Delta t \sum_{k=1}^n \|\nabla \widetilde{\mathbf{e}}_h^k \|_{L^2}^2 \Big)^\frac12 
  \le  \hat{C} ( \Delta t+ h^2 ) . 
  \setlength{\belowdisplayskip}{4pt}  
    \label{convergence-3-14} 
\end{equation}   
By the comparison estimate~\eqref{continuous extension-2}, the convergence estimate~\eqref{convergence-0} becomes valid. In turn, it is observed that the a-priori assumption~\eqref{a priori-1} is recovered at the next time step $t^n$: 
\begin{equation}\label{a priori-6} 
  \setlength{\abovedisplayskip}{4pt} 
  \begin{aligned} 
  & 
\| \mathbf{e}_h^n \|_{L^2_h}  ,\,\,\| \widetilde{\mathbf{e}}_h^n - \mathbf{e}_h^n \|_{L^2_h}\le  \hat{C} ( \Delta t+ h^2 ) \le \Delta t^{1 - \frac{\epsilon_0}{8}} 
 + h^{2 - \frac{\epsilon_0}{4}} , 
\\
  &    
 \| \nabla \widetilde{\mathbf{e}}_h^n \|_{L^2}  \le \frac{\hat{C} ( \Delta t + h^2) }{\gamma^\frac12 \Delta t^\frac12} 
 \le \Delta t^{\frac12 - \frac{\epsilon_0}{8}} + h^{1 - \frac{\epsilon_0}{4}} , 
\end{aligned}  
   \setlength{\belowdisplayskip}{4pt}
\end{equation}
under the scaling law $h^2 \le \Delta t \le h^{\epsilon_0}$, provided that $\Delta t$ and $h$ are sufficiently small. Finally, combining the comparison estimate \eqref{continuous extension-2} (between the discrete and continuous norms) and the estimates for the interpolation operator, we complete the proof of Theorem \ref{thm1error}.
\end{proof}
\begin{remark}
    The error analysis can also be naturally extended to general acute triangular meshes. With the aid of the formula 
    \begin{equation}
        \int_{\triangle ABC}\nabla u\cdot\nabla v = \alpha _1 \big(\frac{u_A -u_B}{|AB|}\big)\big(\frac{v_A -v_B}{|AB|}\big) + \alpha_2 \big(\frac{u_B -u_C}{|BC|}\big)\big(\frac{v_B -v_C}{|BC|}\big) + \alpha_3 \big(\frac{u_C - u_A}{|AC|}\big)\big(\frac{v_C -v_A}{|AC|}\big),
    \end{equation}
    where \(\alpha_1,\,\,\alpha_2,\) and \(\alpha_3\) are constants dependent only on the shape of the triangle for any acute triangular $\triangle ABC \in \mathcal{K}_h$. By combining this representation with the norm equivalence between difference gradients and linear finite elements, an analogous error analysis could be directly carried out within the inner product framework.
\end{remark}
\begin{table}[htbp]
	\centering
	\caption{Error and convergence rates for the magnetization in $L_h^2$ norm on uniform grid with $\Delta t = h^2$}
	\label{table1}
	\small
	\begin{tabular}{c|c|c|c|c} \hline\hline
		$N_x\times N_y$&$||\mathbf{e}_h^n||_{L^{\infty}(L_h^{2})}$ &$Rate$
		&$||\widetilde{\mathbf{e}}_h^n||_{L^\infty(L_h^2)}$ &$Rate$  \\ \hline
		$5\times 5$ &2.40  &---       &2.39    &---          \\
		$10\times 10$ &6.19E-1  &1.96     &6.19E-1     &1.95         \\
		$20\times 20$ &1.57E-1  &1.97    &1.57E-1      &1.97        \\
		$40\times 40$ &3.94E-2  &2.00     &3.94E-2     &2.00        \\
		$80\times 80$ &9.85E-3  &2.00     &9.85E-3      &2.00         \\
		\hline\hline
	\end{tabular}
\end{table}
\begin{table}[htbp]
	\centering
	\caption{Error and convergence rates for the magnetization in $L_h^2$ norm on uniform grid with $\Delta t = h$}
	\label{table2}
	\small
	\begin{tabular}{c|c|c|c|c} \hline\hline
		$N_x\times N_y$&$||\mathbf{e}_h^n||_{L^{\infty}(L_h^{2})}$ &$Rate$
		&$||\widetilde{\mathbf{e}}_h^n||_{L^\infty(L_h^2)}$ &$Rate$  \\ \hline
		$5\times 5$ &3.46  &---       &3.36      &---          \\
		$10\times 10$ &1.90  &0.87     &1.88E-1     &0.84         \\
		$20\times 20$ &8.65E-1  &1.13    &8.64E-1      &1.13        \\
		$40\times 40$ &4.32E-1  &1.00     &4.32E-1     &1.00        \\
		$80\times 80$ &2.19E-1  &0.98     &2.19E-2      &0.98         \\
		\hline\hline
	\end{tabular}
\end{table}

\begin{table}[htbp]
	\centering
	\caption{Error and convergence rates for the magnetization in $L_h^2$ norm on quasi-uniform grid with $\Delta t = h^2$}
	\label{table3}
	\small
	\begin{tabular}{c|c|c|c|c} \hline\hline
		$N_x\times N_y$&$||\mathbf{e}_h^n||_{L^{\infty}(L_h^{2})}$ &$Rate$
		&$||\widetilde{\mathbf{e}}_h^n||_{L^\infty(L_h^2)}$ &$Rate$  \\ \hline
		$5\times 5$ &2.30  &---       &2.29    &---          \\
		$10\times 10$ &6.20E-1  &1.89     &6.20E-1     &1.88         \\
		$20\times 20$ &1.57E-1  &1.98   &1.57E-1      &1.98        \\
		$40\times 40$ &3.95E-2  &1.99     &3.95E-2     &1.99        \\
		$80\times 80$ &9.85E-3  &2.00     &9.85E-3      &2.00         \\
		\hline\hline
	\end{tabular}
\end{table}
\begin{table}[htbp]
	\centering
	\caption{Error and convergence rates for the magnetization in $L_h^2$ norm on quasi-uniform grid with $\Delta t = h$}
	\label{table4}
	\small
	\begin{tabular}{c|c|c|c|c} \hline\hline
		$N_x\times N_y$&$||\mathbf{e}_h^n||_{L^{\infty}(L_h^{2})}$ &$Rate$
		&$||\widetilde{\mathbf{e}}_h^n||_{L^\infty(L_h^2)}$ &$Rate$  \\ \hline
		$5\times 5$ &3.38  &---       &3.29      &---          \\
		$10\times 10$ &1.91  &0.83     &1.89E-1     &0.80         \\
		$20\times 20$ &8.65E-1  &1.14    &8.63E-1      &1.13        \\
		$40\times 40$ &4.32E-1  &1.00     &4.32E-1     &1.00        \\
		$80\times 80$ &2.19E-1  &0.98     &2.19E-2      &0.98         \\
		\hline\hline
	\end{tabular}
\end{table}

\begin{figure}[htbp]
	\centering
    \subfigure[Example 2]
    {\includegraphics[width=0.4\linewidth, trim = {0cm 0cm 0cm 0cm}, clip]{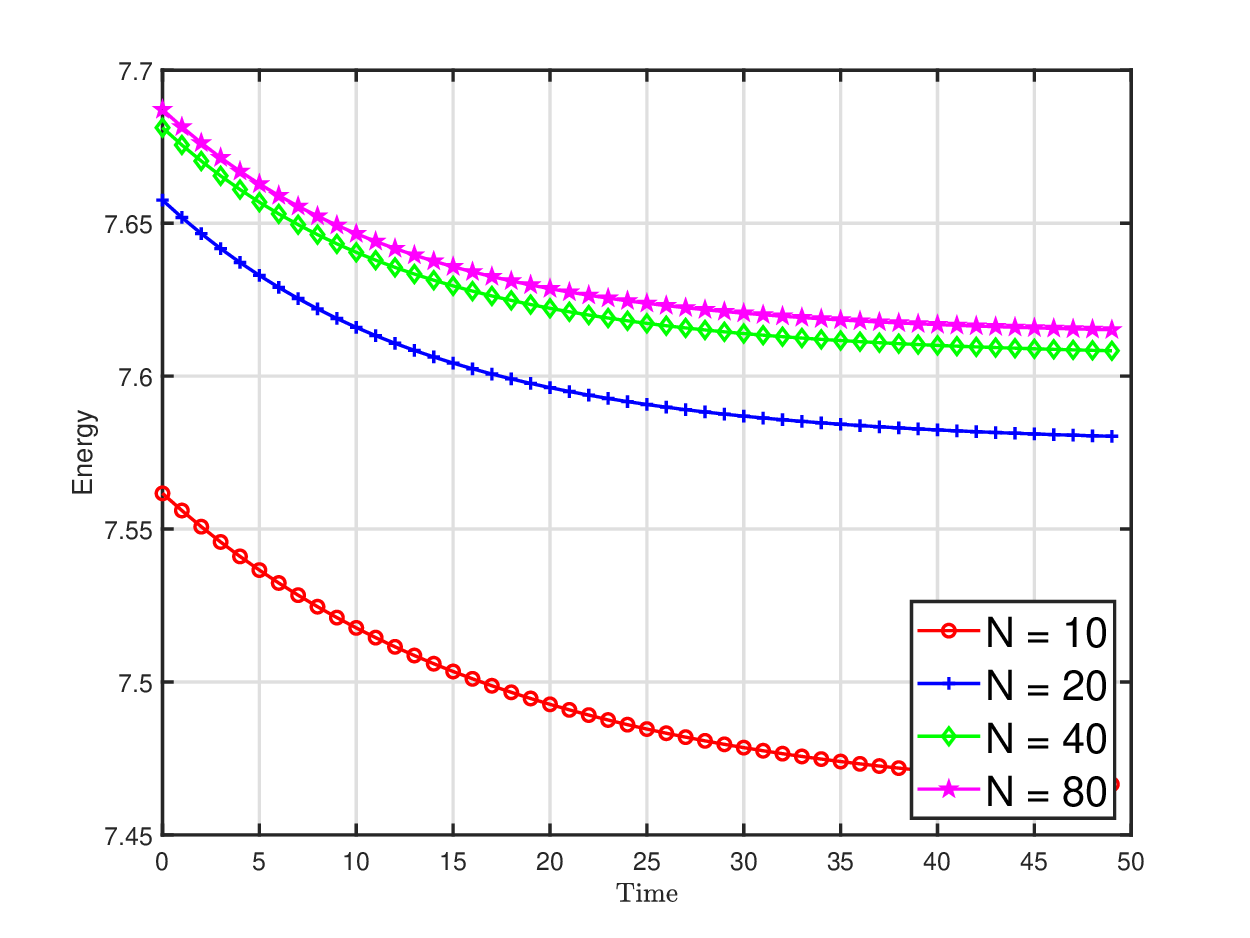}}
    \subfigure[Example 3]
    {\includegraphics[width=0.4\linewidth, trim = {0cm 0cm 0cm 0cm}, clip]{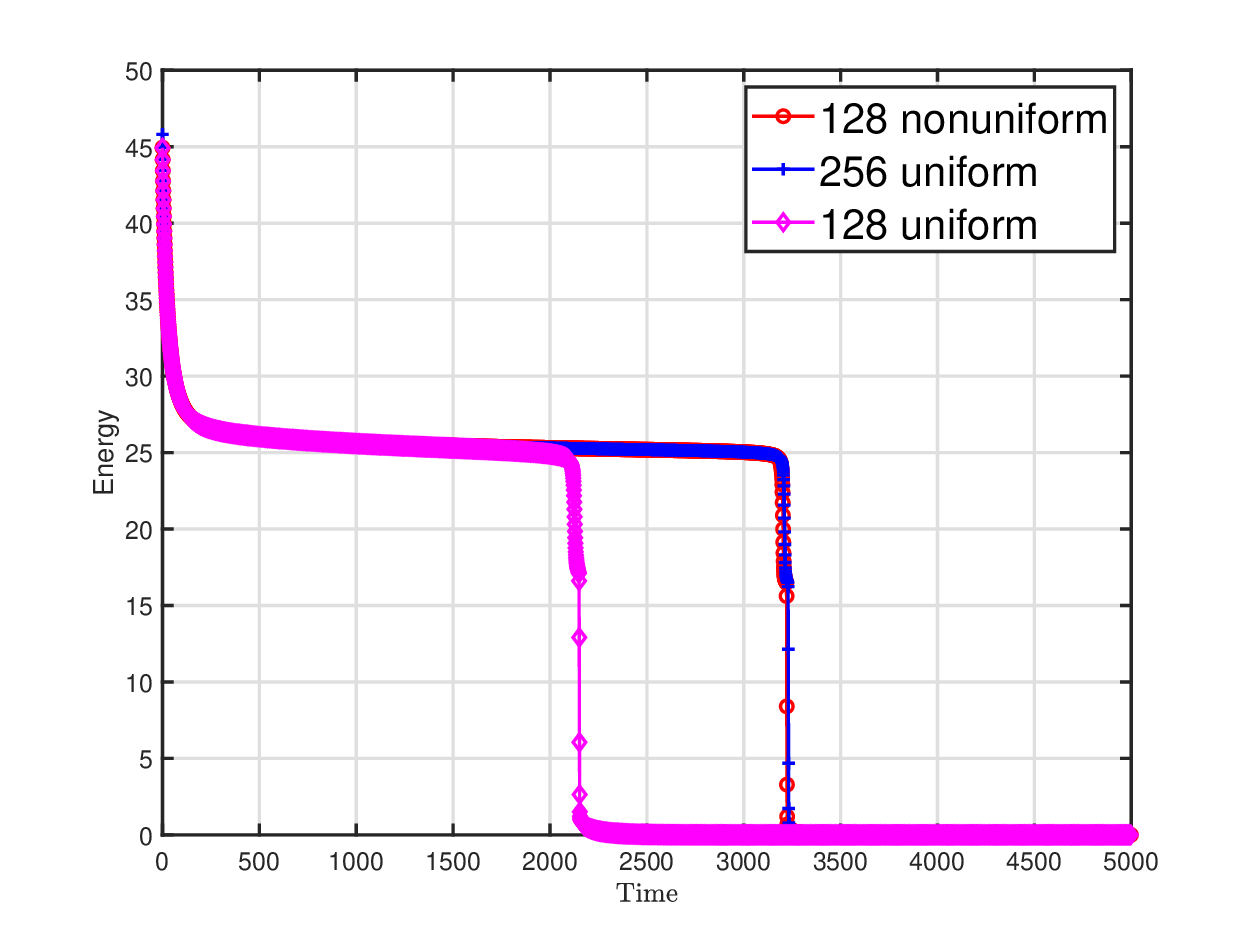}}
	\caption{Evolutions of original energy $\displaystyle \|\nabla\mathbf{m}_h^n\|^2_{L^2}$ for Examples 2 and 3.}
	\label{fig1}
\end{figure}

\begin{figure}[htbp]
	\centering
    \subfigure[Central refinement mesh]
	{\includegraphics[width=0.35\linewidth, trim = {0cm 0cm 0cm 0cm}, clip]{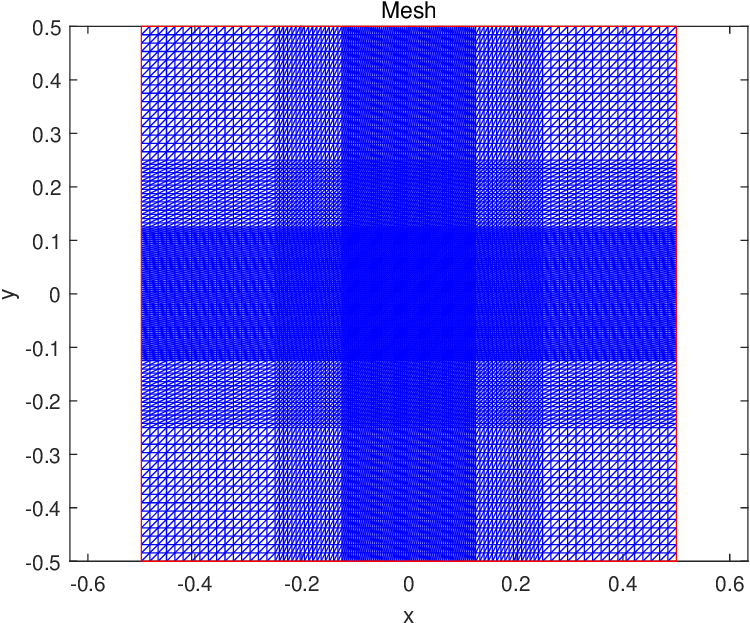}}
	\subfigure[Uniform mesh]
	{\includegraphics[width=0.35\linewidth, trim = {0cm 0cm 0cm 0cm}, clip]{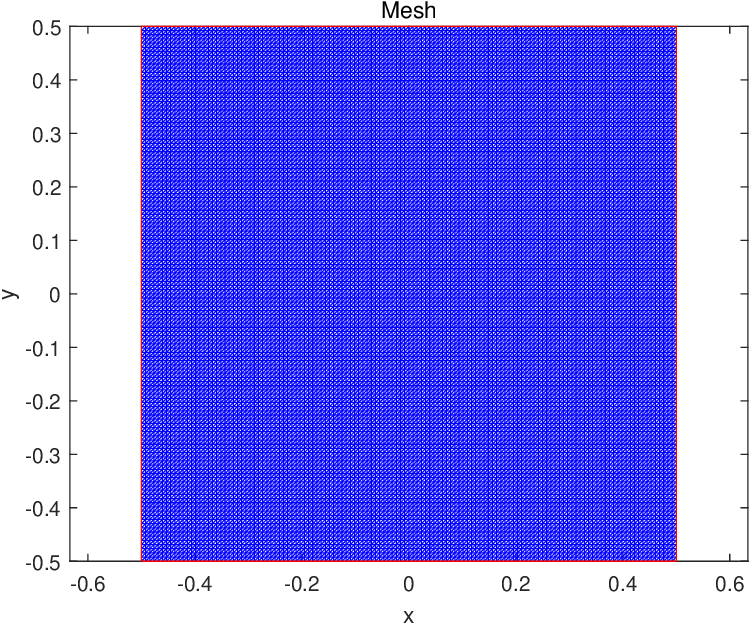}}
	\caption{Central refinement v.s. Uniform mesh discretization}
	\label{fig6}
\end{figure}

\begin{figure}[htbp]
	\centering
    \subfigure[$T=0$]
	{\includegraphics[width=0.3\linewidth, trim = {0cm 0cm 0cm 0cm}, clip]{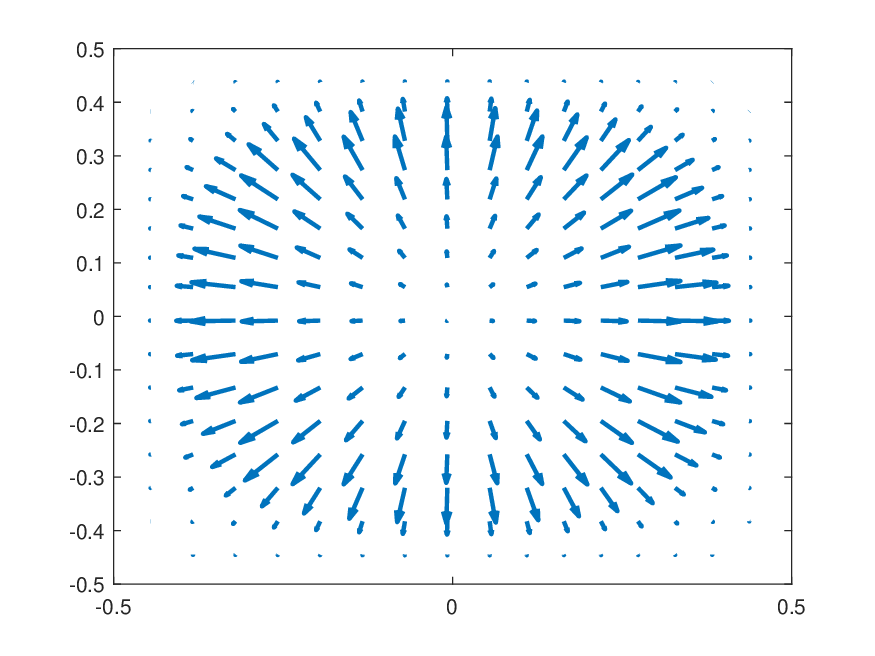}}
	\subfigure[$T=0.06$]
	{\includegraphics[width=0.3\linewidth, trim = {0cm 0cm 0cm 0cm}, clip]{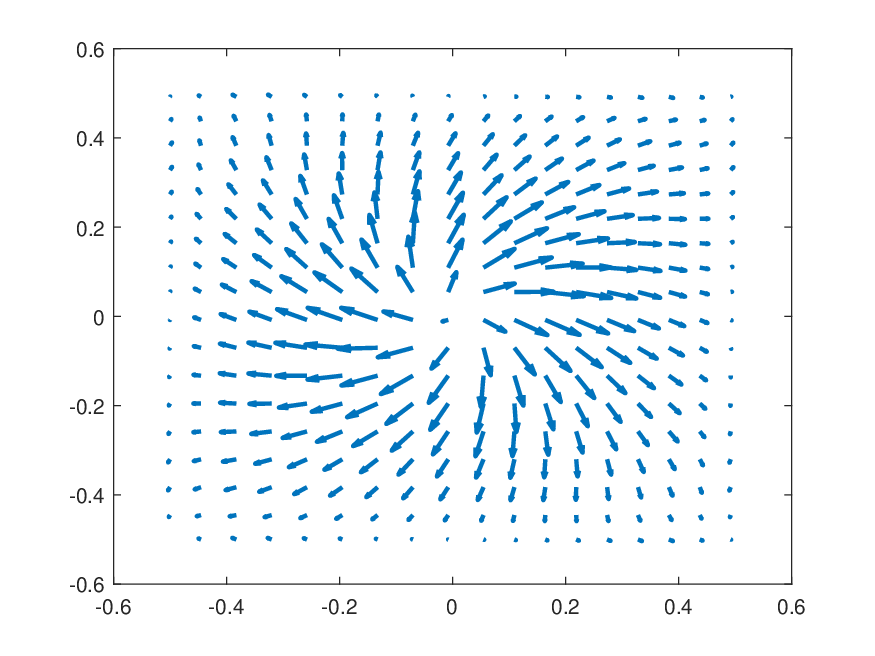}}
	\subfigure[$T=0.15$]
	{\includegraphics[width=0.3\linewidth, trim = {0cm 0cm 0cm 0cm}, clip]{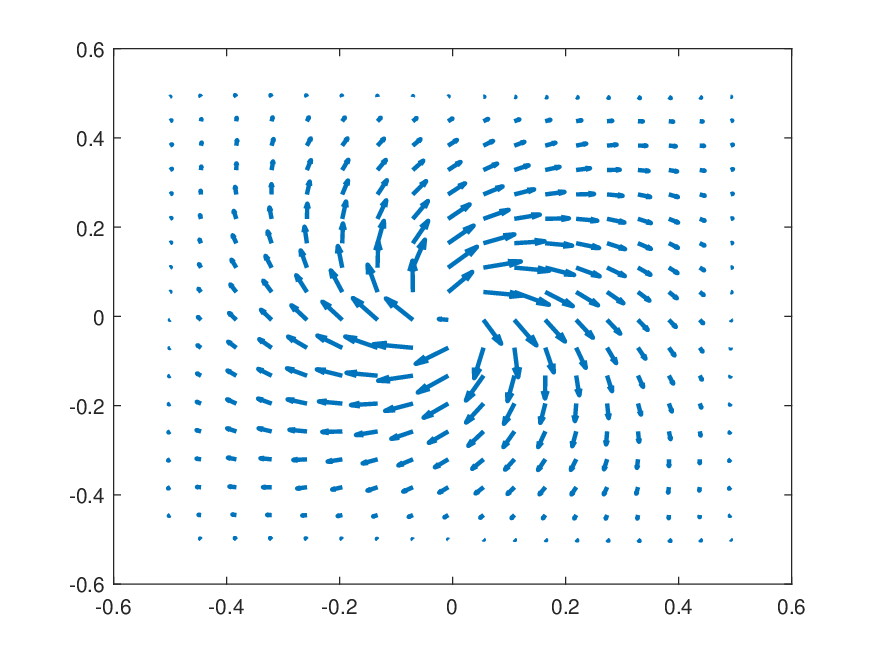}}\\
	\subfigure[$T=0.30$]
	{\includegraphics[width=0.3\linewidth, trim = {0cm 0cm 0cm 0cm}, clip]{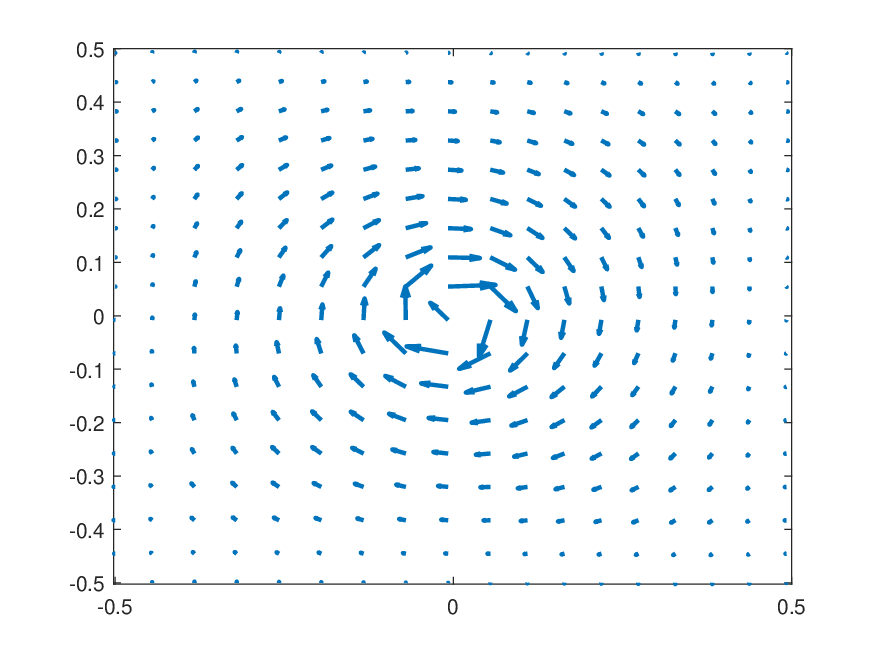}}
	\subfigure[$T=0.32$]
	{\includegraphics[width=0.3\linewidth, trim = {0cm 0cm 0cm 0cm}, clip]{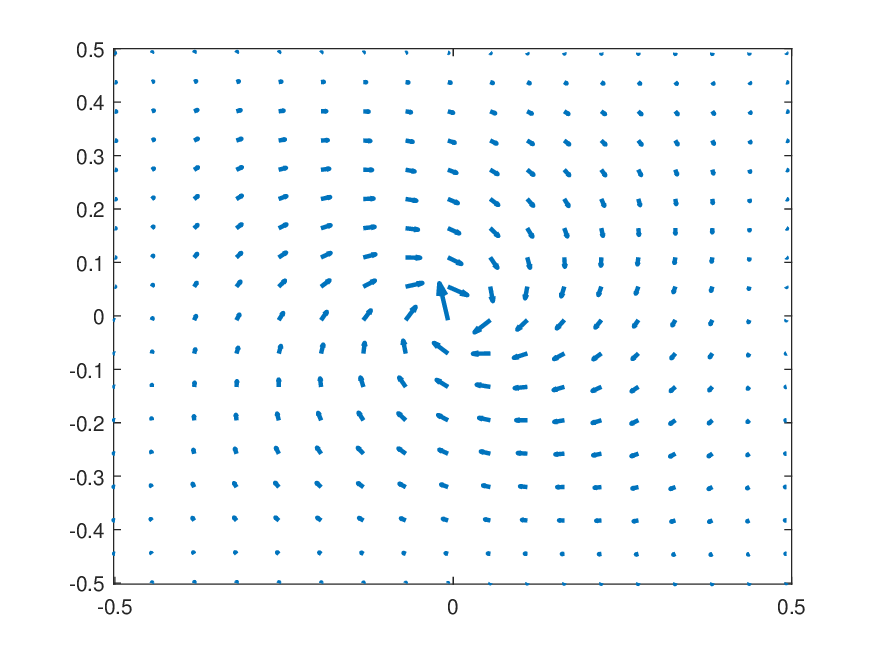}}
    \subfigure[$T=0.35$]
	{\includegraphics[width=0.3\linewidth, trim = {0cm 0cm 0cm 0cm}, clip]{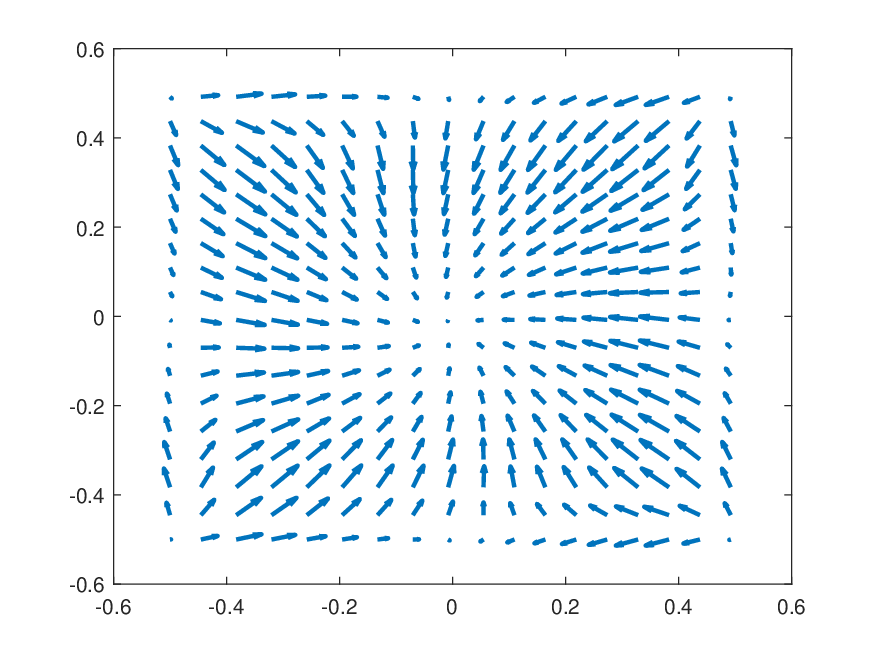}}
	\caption{Numerical magnetization $\mathbf{m}^n$ (projected on $x_1x_2$-plane) for Example 3.}
	\label{fig2}
\end{figure}

\begin{figure}[htbp]
	\centering
    \subfigure[$T=0$]
	{\includegraphics[width=0.3\linewidth, trim = {0cm 0cm 0cm 0cm}, clip]{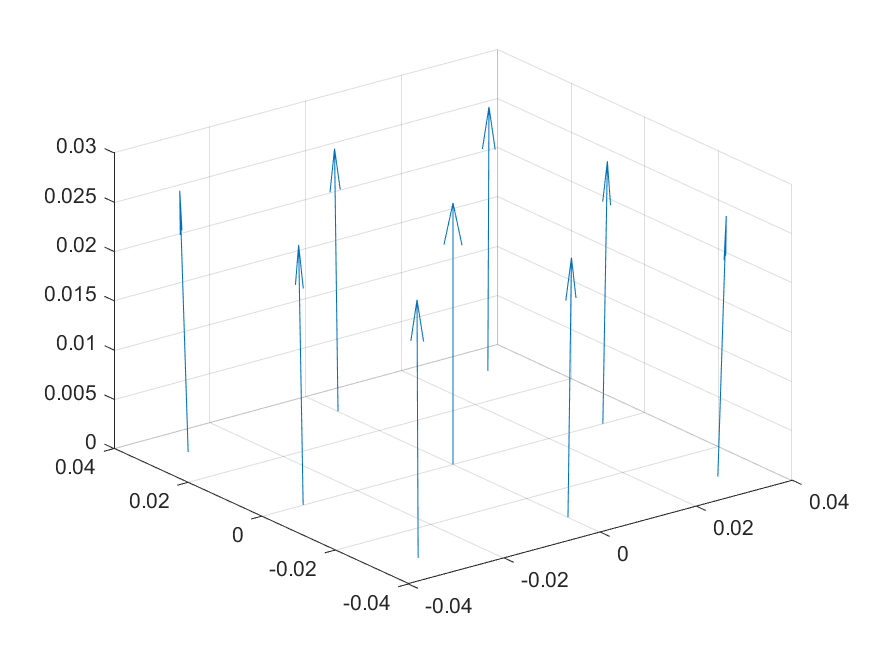}}
	\subfigure[$T=0.06$]
	{\includegraphics[width=0.3\linewidth, trim = {0cm 0cm 0cm 0cm}, clip]{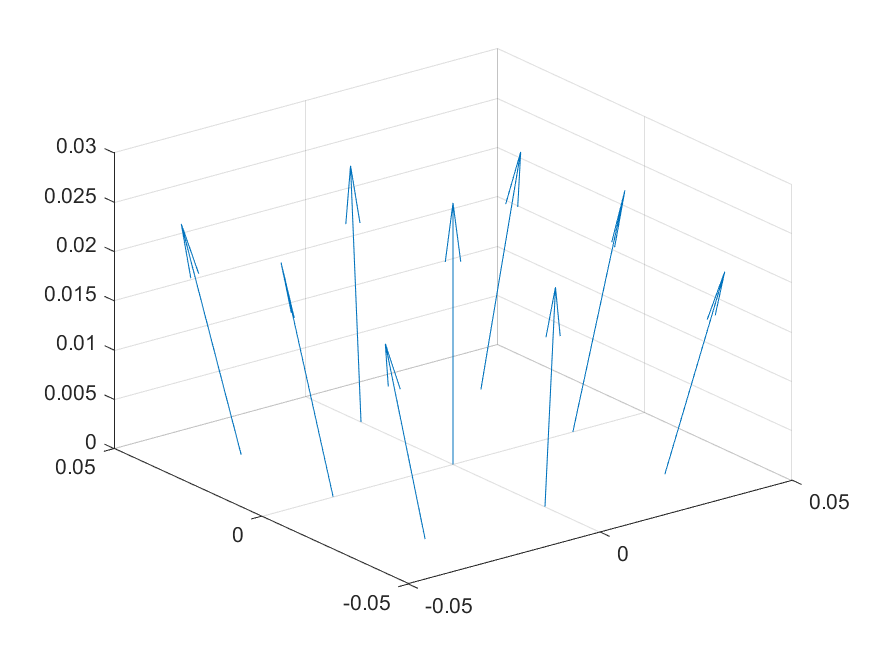}}
	\subfigure[$T=0.15$]
	{\includegraphics[width=0.3\linewidth, trim = {0cm 0cm 0cm 0cm}, clip]{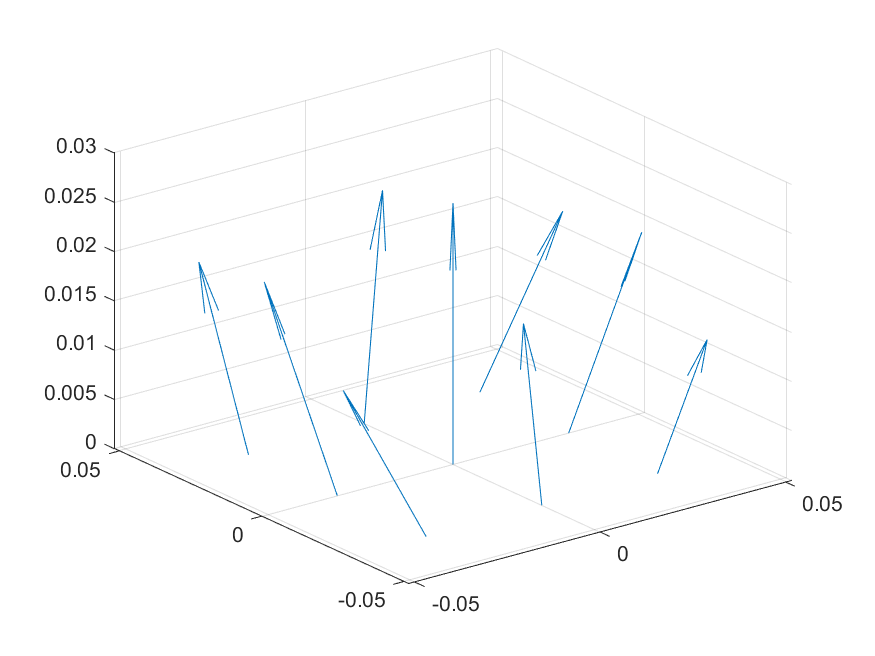}}\\
	\subfigure[$T=0.30$]
	{\includegraphics[width=0.3\linewidth, trim = {0cm 0cm 0cm 0cm}, clip]{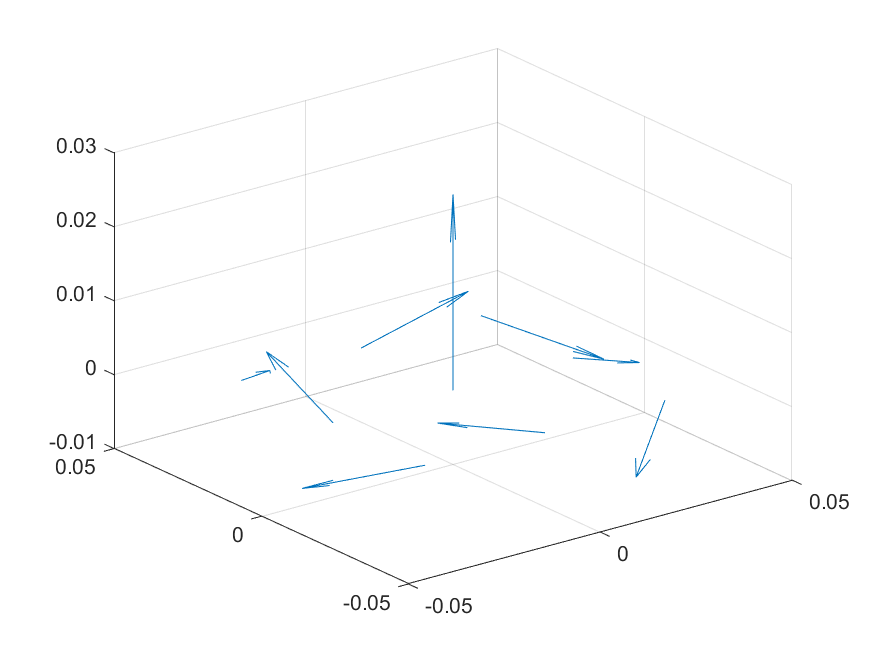}}
	\subfigure[$T=0.32$]
	{\includegraphics[width=0.3\linewidth, trim = {0cm 0cm 0cm 0cm}, clip]{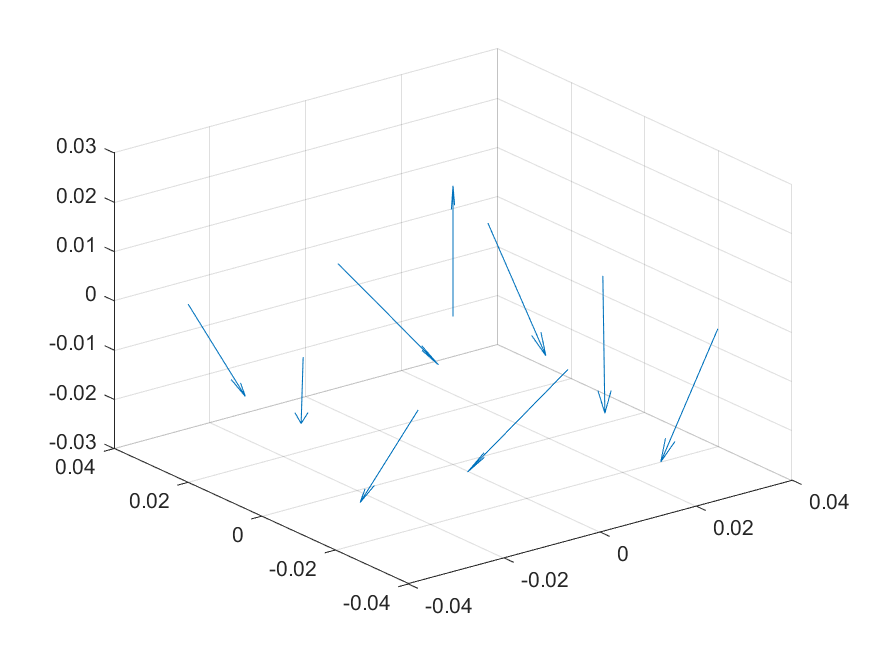}}
    \subfigure[$T=0.35$]
	{\includegraphics[width=0.3\linewidth, trim = {0cm 0cm 0cm 0cm}, clip]{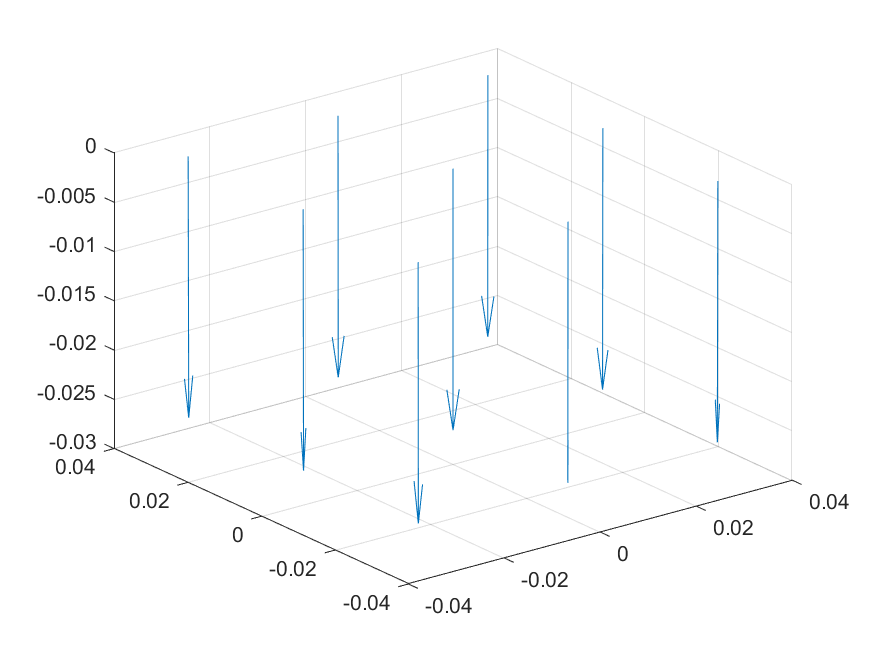}}
	\caption{Numerical magnetization $\mathbf{m}^n$ around the origin for Example 3.}
	\label{fig3}
\end{figure}

\section{Numerical results} 
In this section, we present several numerical experiments to validate the theoretical analysis.

\subsection{Convergence test} The convergence rate is tested for the LLG equation with an external force. The exact solution is set as 
\begin{equation*}
   \setlength{\abovedisplayskip}{4pt} 
    \left\{
    \begin{array}{l}
         m_e^x(x,y,t) = \sin(t+x)\cos(t+y),  \\
         m_e^y(x,y,t) = \cos(t+x)\cos(t+y),  \\
         m_e^z(x,y,t) = \sin(t+y).
    \end{array}  \right. 
    \setlength{\belowdisplayskip}{4pt} 
\end{equation*}
It is clear that the above-defined $\mathbf{m}_e$ satisfies an exact node-wise length preservation, $|\mathbf{m}| =1$. We take $\beta = 1,\,\gamma=1,\,T=1$ and $\Omega = (0,\,2\pi)^2$, and periodic boundary conditions are used. The initial spatial partition is taken as a $5\times5$ grid, and the time step size is refined as $\Delta t = 1/N_x^2 = 1/N_y^2$, to demonstrate both the temporal and spatial convergence rates. Tables \ref{table1}-\ref{table4} clearly indicate the second-order accuracy in space and first-order convergence rate in time, respectively, with  $\Delta t = h^2$ and $\Delta t = h$ in the discrete $L^2_h$ norm. These numerical results are in good agreement with the theoretical predictions in Theorem \ref{thm1error}. Tables \ref{table3}-\ref{table4} present the numerical results on quasi-uniform grids, where the grid ratio $h_{max}/h_{min} \leq 3$ satisfies the quasi-uniform mesh condition.

\subsection{Energy dissipation test} In this simulation, we verify that the constructed scheme \eqref{discretescheme} is unconditionally dissipative in terms of the original energy. We set 
\begin{equation*} 
  \setlength{\abovedisplayskip}{4pt}  
  T = 1,\,\gamma = 1,\,\beta = 1,\,\Omega = (0,\,2\pi)^2,
  \setlength{\belowdisplayskip}{4pt} 
\end{equation*} 
and use the following initial profile 
\begin{equation*}
   \setlength{\abovedisplayskip}{4pt}  
    \left\{
    \begin{array}{l}
         m^x_e(x,y,0) = \sin(x)\cos(y),  \\
         m^y_e(x,y,0) = \cos(x)\cos(y),  \\
         m^z_e(x,y,0) = \sin(y). 
    \end{array} \right. 
    \setlength{\belowdisplayskip}{4pt} 
\end{equation*}

The time evolution of the original energy functional, given by  $ \int_\Omega|\nabla\mathbf{m}_h|^2d\mathbf{x}$, with time step $\Delta t =0.01$ and spatial mesh sizes, is presented in Figure \ref{fig1} (a). It is observed that the original energy decreases monotonically for different spatial mesh sizes, indicating a dissipative behavior.
{\color{black}
\subsection{Phenomenon of blowup}
In this subsection, we investigate the possible blowup of the LLG equation with certain smooth initial data, as given by \cite{an2021optimal, bartels2008numerical, chen1998evolution}. Set $\Omega = (-1/2,\,1/2)^2$, and let the initial data $\mathbf{m}^0$ be defined by
\begin{equation*}
   \setlength{\abovedisplayskip}{4pt}  
    \mathbf{m}^0(\mathbf{x}) = \left\{
    \begin{array}{ll}
         (0,\,0,\,-1) , &\,\forall\,|\mathbf{x}|\geq1/2,  \\
        \displaystyle (\frac{2x_1A}{A^2+|\mathbf{x}|^2},\,\frac{2x_2A}{A^2+|\mathbf{x}|^2},\,\frac{A^2 - |\mathbf{x}|^2}{A^2+|\mathbf{x}|^2}) , &\,\forall\,|\mathbf{x}|\leq1/2, 
    \end{array} \right. 
   \setlength{\belowdisplayskip}{4pt} 
\end{equation*}
with $A = (1 - 2|\mathbf{x}|)^4$. We take the the parameters $\beta = 1,\,\gamma = 1$ in the LLG equation \eqref{originalmodel}.

The LLG equation is solved by the proposed scheme \eqref{discretescheme} on a quasi-uniform mesh with $N_x =N_y= 128$ and $\Delta t = 10^{-4}$. The orthogonal projection of the vector field $\mathbf{m}^{n+1}$ on the $x_1x_2$-plane, and close-up pictures of $\mathbf{m}^{n+1}$ near the origin at a sequence of time instants, $t = 0,\,0.06, 0.15,\,0.30,\,0.32,\,0.35$, are displayed in Figures. \ref{fig2}-\ref{fig3}. It is observed that $\mathbf{m}^{n+1}$ preserves $(0,\,0,\,1)^T$ at the origin and gradually turns down to $(0,\,0,\,-1)^T$ near the origin, which is consistent with the blowup phenomenon presented in \cite{an2021optimal, bartels2008numerical}. 

Unlike finite difference methods that are typically restricted to uniform meshes in \cite{li2026stability}, the proposed scheme in this work could be naturally extended to quasi-uniform mesh discretization. To further validate the benefits of the quasi-uniform mesh implementation, we take the numerical results acquired on a uniform \(256\times256\) mesh in \cite{li2026stability} as the reference solution and quantitatively compare the energy evolution curves produced by the centrally refined quasi-uniform \(128\times128\) mesh and the standard uniform \(128\times128\) mesh, where comparisons of different mesh discretizations are illustrated in Figure \ref{fig6}. As shown in Figure \ref{fig1}(b), monotonic energy decay is observed for all mesh configurations. The energy profile of the centrally refined quasi-uniform \(128\times128\) mesh matches the reference solution considerably well. In addition, the quasi-uniform mesh scheme achieves a computational overhead comparable to that of the uniform \(128\times128\) mesh and significantly lower than that of the high-resolution uniform \(256\times256\) mesh. These observations firmly validate the superiority and efficiency of extending the developed algorithm to quasi-uniform meshes.
}
\section{Conclusion}
In this article, we propose and analyze a linear mass-lumped finite element numerical method for the Landau-Lifshitz-Gilbert equation, which is straightforward to implement and enforces the non-convex constraint $|\mathbf{m}^n_h| = 1$ at all finite element nodes. Besides the tangent-plane finite element method, which is distinct from classical finite element approaches, to the best of our knowledge, the proposed scheme is currently the only classical finite element method that enforces both the unit-length constraint and unconditional original energy dissipation simultaneously. Based on the equivalent weak formulation, we come up with an innovative technique to combine the finite element method with finite difference analysis, by employing interpolation operators to treat the highly nonlinear term. Such an approach enables us to derive an optimal rate convergence analysis and error estimate as well. Theoretically, this numerical method overcomes the difficulty that classical finite element methods cannot preserve node-wise properties, and establishes the result of optimal-order convergence on triangular meshes. 
Developing higher-order finite element methods based on renormalization techniques that maintain these two key properties simultaneously remains a significant challenge and will constitute a primary direction for our future research.

\bibliographystyle{siamplain}
\bibliography{ref}
\end{document}